\documentclass[12pt]{amsart}

\newcommand{\ooo}[1]{{\color{red} \sf $\clubsuit\clubsuit\clubsuit$ [#1]}}
\makeatletter
\def\@citecolor{blue}
\def\@urlcolor{blue}
\def\@linkcolor{blue}
\makeatother
\usepackage{tikz}
\usetikzlibrary{calc,intersections,through,backgrounds,patterns}
\pgfdeclarepatternformonly{my crosshatch dots}{\pgfqpoint{-1pt}{-1pt}}{\pgfqpoint{5pt}{5pt}}{\pgfqpoint{6pt}{6pt}}%
{
    \pgfpathcircle{\pgfqpoint{0pt}{0pt}}{.5pt}
    \pgfpathcircle{\pgfqpoint{3pt}{3pt}}{.5pt}
    \pgfusepath{fill}
}

\usepackage{enumerate}
\usepackage[all]{xy}  
\usepackage{graphicx}
\usepackage[headings]{fullpage}
\usepackage{amsmath,amsthm,amssymb,amsfonts,latexsym,mathrsfs}
\usepackage{comment,verbatim}
\usepackage{mathdots}
\usepackage{xcolor}
\usepackage{ifthen}
\usepackage{wrapfig}
\usepackage[colorlinks,pagebackref=true,pdftex]{hyperref}
\usepackage[normalem]{ulem}

\makeatletter

\@addtoreset{equation}{section}
\def\theequation{\thesection.\@arabic \c@equation}
\def\@citecolor{blue}
\def\@urlcolor{blue}
\def\@linkcolor{blue}
\def\theenumi{\@roman\c@enumi}

\makeatother

\theoremstyle{definition}
\newtheorem{Theorem}{Theorem}
\newtheorem{theorem}{Theorem}[section]
\newtheorem{lemma}[theorem]{Lemma}
\newtheorem{corollary}[theorem]{Corollary}
\newtheorem{proposition}[theorem]{Proposition}
\newtheorem*{claim*}{Claim}

\theoremstyle{definition}
\newtheorem{remark}[theorem]{Remark}

\newtheorem{remarks}[theorem]{Remarks}

\newtheorem{example}[theorem]{Example}

\newtheorem{definition}[theorem]{Definition}

\newcommand{\m}{\mathfrak{m}}

\def\NZQ{\mathbb}               % the font for N,Z,Q,R,C
\def\NN{{\NZQ N}}
\def\ZZ{{\NZQ Z}}

\def\frk{\mathfrak}               % font for "Fraktur"
\def\aa{{\frk a}}

\def\mm{{\frk m}}
\def\nn{{\frk n}}
\def\pp{{\frk p}}

 \newcommand{\ul}[1]{
\underline{#1}}
\newcommand{\ov}[1]{\overline{{#1}}}

\def\opn#1#2{\def#1{\operatorname{#2}}} % to make operators
\opn\chara{char}
\opn\length{\ell}
\opn\projdim{proj\,dim}
\opn\depth{depth}
\opn\reg{reg}
\opn\lreg{lreg}
\opn\sat{^{\hspace{-.05cm}\rm sat}}
\opn\lex{^{\hspace{-.05cm}lex}}
\opn\Ker{Ker}
\opn\Coker{Coker}
\opn\Im{Im}
\opn\Hom{Hom}
\opn\Tor{Tor}
\opn\Ext{Ext}
\opn\End{End}
\opn\Aut{Aut}
\opn\id{id}
\opn\GL{GL}
\opn\Ass{Ass}
\opn\Supp{Supp}
\opn\Spec{Spec}
\opn\Min{Min}
\renewcommand{\leq}{\leqslant}
\renewcommand{\geq}{\geqslant}

\let\:=\colon
\let\ov\overline
\let\lra\longrightarrow

\opn\Gin{Gin}

\opn\Hilb{Hilb}
\opn\ini{in}
\opn\End{end}
\opn\Ann{Ann}
\opn\height{ht}

\DeclareMathOperator{\Edepth}{E-depth}
\DeclareMathOperator{\rev}{rev}
\DeclareMathOperator{\gin}{gin}

\title{On the notion of sequentially Cohen-Macaulay modules}

\author[G. Caviglia]{Giulio Caviglia}
\address{Giu\-lio Ca\-vi\-glia - Department of Mathematics -  Purdue University - 150 N. University Street, West Lafayette - 
  IN 47907-2067 - USA} 
\email{gcavigli@math.purdue.edu}
\author[A. De Stefani]{Alessandro De Stefani}
\address{Alessandro De Stefani - Dipartimento di Matematica - Universit{\`a} di Genova - Via Dodecaneso 35 - 16146 Genova - Italy}
\email{destefani@dima.unige.it}
\author[E. Sbarra]{Enrico Sbarra}
\address{Enrico Sbarra - Dipartimento di Matematica - Universit\`a degli Studi di Pisa - Largo Bruno Pontecorvo 5 - 56127 Pisa - Italy}
\email{enrico.sbarra@unipi.it}
\author[F. Strazzanti]{Francesco Strazzanti}
\address{Francesco Strazzanti - Dipartimento di Matematica ``Giuseppe Peano'' - Universit\`a degli Studi di Torino - Via Carlo Alberto 10 - 10123 Torino - Italy}
\email{francesco.strazzanti@gmail.com}

\thanks{The first author was partially supported by a grant from the Simons Foundation (41000748, G.C.). The second author was partially supported by the PRIN 2020 project 2020355B8Y ``Squarefree Gr{\"o}bner degenerations, special varieties and related topics''. The third author was partially supported by the PRIN project 2017SSNZAW004 ``Moduli Theory and Birational Classification"  of Italian MIUR}

\date{\today}
\begin{document}
%\today

\begin{abstract}
In this survey paper we first present the main properties of sequentially Cohen-Macaulay modules. Some  basic examples are provided to help the reader with quickly getting acquainted with this topic.
We then discuss two generalizations of the notion of sequential Cohen-Macaulayness which are inspired by a theorem of J\"urgen Herzog and the third author. 
\end{abstract}

\maketitle

\section*{Introduction} 
This survey note, which is dedicated to the work of J\"urgen Herzog on the topic, cannot possibly be complete: the notion of sequentially Cohen-Macaulay module has been central in many papers in the literature from the late 90's, starting maybe with \cite{D96}. On the other hand, in the late 90's Herzog's research activity was feverish as he counted very many visitors and collaborators. At the same time the distribution of preprints in the form of postscript files over the internet expedited the dissemination of mathematical ideas. We thus must apologize in advance that our reference list is doomed to be incomplete.

It is in 1997 that Herzog, together with who would become his top co-author, Takayuki Hibi, published a paper on simplicial complexes \cite{HH97}, immediately followed by another series of papers of the two authors together with Annetta Aramova \cite{AHH97, AHH97-2, AHH98}, where numerical problems about simplicial complexes were addressed through the study of Hilbert functions, Gr\"obner bases techniques and generic initial ideals, see also \cite{ADH98, AH97, DH98}. In 1999, another influential paper authored by Herzog and Hibi, is published, \cite{HH99}: they introduce and study a new class of ideals, called componentwise linear. 
Componentwise linear Stanley-Reisner ideals $I_\Delta$ are characterized as those for which the pure $i$-th skeleton of the Alexander dual of $\Delta$ is Cohen-Macaulay for every $i$.
Thanks to \cite{D96}, this means that $I_\Delta$  is componentwise linear if and only its  Alexander dual is sequentially Cohen-Macaulay, see also \cite{HRW99}.
In this way, the authors were also able to generalize a well-known result of Eagon and Reiner, which says that the Stanley-Reisner ideal of a simplicial complex has a linear resolution if and only if its Alexander dual is Cohen-Macaulay. 
All the ideas behind the proofs of these facts led also to another fundamental result, which is the main theorem of \cite{HS01} and provides a somewhat unexpected characterization of graded sequentially Cohen-Macaulay modules over a polynomial ring $R$ in terms of the Hilbert function of the local cohomology modules: 

\begin{Theorem} \label{HS Intro}
Let $M$ be a finitely generated graded $R$-module, and let $M \cong F/U$ a free graded presentation of $M$. Then, $F/U$ is sequentially Cohen-Macaulay  if and only if $\Hilb(H^i_{\mm}(F/U))=\Hilb(H^i_{\mm}(F/\Gin(U)))$ for all $i \geq 0$. 
\end{Theorem}
\noindent
Here $\Gin(U)$ denotes the generic initial module of $U$ with respect to the degree reverse lexicographic order.

\smallskip 

This paper is organized as follows. In the first section we introduce the definition of sequentially Cohen-Macaulay modules according to Stanley \cite{St96}, and discuss some basic properties and examples. In Section 2 we present some of the main characterizations, or equivalent definitions, of sequential Cohen-Macaulayness, by recalling Schenzel's results on the dimension filtration of a module, cf. Propositions \ref{sche02}, \ref{CMFvsDF} and Theorem \ref{thm sCM and CMf},  and Peskine's characterization in terms of deficiency modules, cf. Theorem \ref{peskine}. In Section 3, we recollect the definition of partial sequential Cohen-Macaulayness introduced by the third and the fourth author in \cite{SS17}, by requiring that only the queue of the dimension filtration is a Cohen-Macaulay filtration, see Definition \ref{Def i-sCM} and some basic properties in the graded case in Lemmata \ref{Afshin} and \ref{lemma SbSt}. The next two results, Lemma \ref{lemma finite length} and Proposition \ref{Ale}, are dedicated to filling a gap in the proof of a fundamental lemma in \cite{SS17}, and we present the first of our generalizations of Theorem \ref{HS Intro} in Theorem \ref{partiallySCM}. We start Section 4 by recalling the definition of E-depth, as introduced by the first and second author in \cite{CD21}, see Definition \ref{Def E-depth}. E-depth measures how much depth the deficiency modules of a finitely generated standard graded module $M$ have altogether, and sequential Cohen-Macaulayness can be detected by E-depth as observed in Remark \ref{Remark E-depth}. Some interesting homological properties of E-depth, especially in connection with strictly filter regular elements, are shown in Proposition \ref{prop Edepth}. In Definition \ref{def partial gin} we introduce the other crucial notion for the following, what we might call partial generic initial ideal, making use of a special partial revlex order. Finally, by means of Proposition \ref{prop gin and grading}, we prove in Theorem \ref{main thm Edepth} the main result of this section, which can be seen as another generalization of Theorem \ref{HS Intro}.

\smallskip

To J\"urgen Herzog, a bright example of mathematician and teacher.

\section{Sequentially Cohen-Macaulay modules}

Throughout the paper, let $(R,\m,k)$ denote either a Noetherian local ring with maximal ideal $\m$ and residue field $k=R/\m$, or a standard graded $k$-algebra $R=\bigoplus_{i \geq 0} R_i$, with $R_0=k$ and $\m = \bigoplus_{i>0} R_i$. In the second case, every $R$-module we consider will be a graded $R$-module, and homomorphisms will be graded of degree zero. 

One of the features that makes the Cohen-Macaulay property significant is its characterization in terms of the vanishing and non-vanishing of local cohomology: for a $d$-dimensional finitely generated module $M$ with $t=\depth M$, it holds that $H^i_\mm(M)=0$ for all $i<t$ and $i>d$;  also,  $H^t_\mm(M) \neq 0$ and $ H^d_\mm(M)\neq 0$. These results are originally due to Grothendieck, cf. \cite{BH93}, Theorem 3.5.8, Corollary 3.5.9, Corollary 3.5.11.a) and b). As a consequence, $M$ is a Cohen-Macaulay module if and only if $H^i_\mm(M)=0$ for all $i\ne d$.

Let $\dim(R)=n$. When $R$ is Cohen-Macaulay and has a canonical module $\omega_R$, by duality,  the above conditions can be checked on the Matlis dual of the local cohomology modules, i.e., on the modules $\Ext^{n-i}_R(M,\omega_R)$. With the above notation, we then have that $\Ext_R^{n-t}(M,\omega_R)$ and $\Ext_R^{n-d}(M,\omega_R)$ are non-zero, and furthermore that $\Ext^i_R(M,\omega_R)=0$ for all $i<n-d$ and all $i>n-t$. Moreover, $M$ is Cohen-Macaulay if and only if $\Ext^{n-i}_R(M,\omega_R) = 0$ for all $i \ne d$. Often in the literature the modules $\Ext^{n-i}_R(M,\omega_R)$,\,  $i=0,\ldots,n$,\, are called the {\em deficiency modules of $M$}, since they also measure how far a module is from being Cohen-Macaulay.

We now introduce the class of sequentially Cohen-Macaulay modules, following \cite{St96}, recalling their main properties, and in the next section we discuss some of its equivalent definitions. Our main references here are \cite{HS01} and \cite{Sch99}. 

\begin{definition}
A finitely generated $R$-module $M$ is called {\em sequentially Cohen-Macaulay} if it admits a filtration of submodules $0=M_0\subsetneq M_1\subsetneq \ldots \subsetneq M_r=M$ such that each quotient of the filtration $M_{i}/M_{i-1}$ is Cohen-Macaulay and $\dim (M_{i}/M_{i-1}) < \dim (M_{i+1}/M_i)$ for all $i$. In this case, we say that the above filtration is a {\em sCM filtration} for $M$.
\end{definition}

\begin{example}\label{example sCM}
  \begin{enumerate}[(1)]
\item    It is clear from the definition that if $M \ne 0$ is a Cohen-Macaulay module, then it is sequentially Cohen-Macaulay. In fact $M=M_1\supsetneq M_0 = 0$.

\item  Any $1$-dimensional module $M$ is sequentially Cohen-Macaulay, for if $M$ is not Cohen-Macaulay one can take $M_1 = H^0_\m(M)$ and $M_2=M$.

\item  A domain  $R$  is  sequentially Cohen-Macaulay  if and only if it is Cohen-Macaulay. In fact, any non-zero submodule of $R$ will have dimension $d=\dim(R)$, and the only possible sCM filtration for $R$ is $0=M_0 \subsetneq M_1=R$.

\item On the other hand, if $M$ has a Cohen-Macaulay submodule $M_1$ such that $M/M_1$ is sequentially Cohen-Macaulay with a filtration $0=M_1/M_1 \subsetneq N_1/M_1  \subsetneq \ldots  \subsetneq N_s/M_1=M/M_1$, and $\dim M_1 < \dim N_1/M_1$, then $0\subsetneq M_1\subsetneq N_1\subsetneq\ldots\subsetneq  N_s=M$ is a sCM filtration of $M$.
  \end{enumerate}

\end{example}

The notion of sequentially Cohen-Macaulay modules appears for the first time in the literature in \cite{St96} in the graded setting, in connection with the theory of Stanley-Reisner rings {and simplicial complexes. Later on and independently, in \cite{Sch99}, the notion of Cohen-Macaulay filtered modules has been introduced in the local case; in the same paper, it is proven that the two notions coincide.  Since then, sequentially Cohen-Macaulay modules have been extensively studied especially in connection with shellability and graph theory, see for instance
  \cite{VTV08}
  \cite{HTYZ11}
  \cite{AM15},
  \cite{Go16}, 
  \cite{ABG17},
    and the definition has been extended in other directions, see  \cite{CC07}, \cite{R17}. Section 3 and 4 are devoted to two such extensions, due to the third and fourth author and the first two authors respectively.

      \begin{example}\label{nuovo} Let $M$ be a sequentially Cohen-Macaulay module with sCM filtration $0 = M_0 \subsetneq M_1 \subsetneq \dots \subsetneq M_r=M$ and let $d_i=\dim(M_i/M_{i-1})$ for all $i$.
        \begin{enumerate}[(1)]
        \item Let  $N$ be another sequentially Cohen-Macaulay; then $M \oplus N$ is sequentially Cohen-Macaulay. To see this, let $0 =N_0 \subsetneq N_1 \subsetneq \ldots \subsetneq N_s=N$ be a sCM filtration of $N$ and let $\delta_i= \dim(N_i/N_{i-1})$. For convenience, also let $d_0=\delta_0=-1$. Then,  a filtration whose terms are of the form $M_i\oplus N_j$, where either $d_i=\max\{d_a \mid d_a \leq \delta_j\}$ or $\delta_j=\max\{\delta_a \mid \delta_a \leq d_i\}$, is a sCM  filtration of $M_r\oplus N_s = M\oplus N$.
          
  \item The completion $\widehat{M}$ of $M$ at the maximal ideal $\m$ is a sequentially Cohen-Macaulay $\widehat{R}$-module. To this end, let $M'_i= M_i \otimes_R \widehat{R}$, and by flatness of $\widehat{R}$ we have that  $0 = M'_0 \subsetneq M'_1 \subsetneq \ldots \subsetneq M'_r$ is a sCM filtration of $M \otimes_R \widehat{R} \cong \widehat{M}$.
  \item It can be proven in general  that being sequentially Cohen-Macaulay is preserved by faithfully flat base changes; with some extra effort one can derive that $M$ is a sequentially Cohen-Macaulay $R$-module if and only if $M[|x|]$ is sequentially Cohen-Macaulay as an $R[|x|]$-module, cf \cite[Theorem 6.2]{Sch99}.  
        \end{enumerate}
      \end{example}

      \begin{remark}\label{perdopo} In the above example Part (1),  the converse also holds, cf. \cite[Proposition 4.5]{CN03}, \cite[Proposition 3.2]{TPDA18} or Corollary \ref{sumviceversa}.

        \smallskip
        \noindent
 In Part (2) of the same example,  the vice versa does not hold in general, see for instance \cite[Example 6.1]{Sch99}. It is true though when $R$ is a homomorphic image of a Cohen-Macaulay ring, see  \cite[Corollaries 2.8]{GHS10} and \cite[Corollaries 2.9]{GHS10}, where it is proved that, in such a case $M$ is a sequentially Cohen-Macaulay $R$-module if and only if $\widehat{M}$ is a sequentially Cohen-Macaulay $\widehat{R}$-module.        
      \end{remark}
      
\begin{proposition}\label{sullecoom}
  Let $M$ be a sequentially Cohen-Macaulay $R$-module with sCM filtration $0=M_0 \subsetneq M_1 \subsetneq \ldots \subsetneq M_r=M$. Also let $d_i=\dim(M_{i}/M_{i-1})$, for all  $i \in \{1,\ldots,r\}$; then:

  \begin{enumerate}[(1)]
  \item for all $j \in \ZZ$,  we have that\,\,\,
    $H^j_\m(M) \ne 0$ if and only if  $j \in \{d_1,\ldots,d_r\}$, and  $$H^{d_i}_\mm(M)\cong H^{d_i}_\mm(M_i) \cong H^{d_i}_\mm(M_i/M_{i-1});$$
  \item for all $i \in \{1,\ldots,r\}$ and $j < i$, the modules $M_i$ and $M_i/M_j$ are sequentially Cohen-Macaulay.
  \end{enumerate}

\end{proposition}
\begin{proof}
For Part (1), using the short exact sequences $0 \to M_{i-1} \to M_i \to M_i/M_{i-1} \to 0$ one inductively shows that $\dim(M_i) = d_i$ for all $i \in \{1,\ldots,r\}$, since $d_{i-1}<d_i$. The induced long exact sequences on local cohomology
\[
\cdots\to H^{j-1}_\m(M_i/M_{i-1}) \to H^j_\m(M_{i-1}) \to  H^j_\m(M_i) \to H^j_\m(M_i/M_{i-1}) \to  H^{j+1}_\m(M_{i-1})\to\cdots
\]

\noindent
together with the fact that $M_i/M_{i-1}$ is Cohen-Macaulay of dimension $d_i$, imply  that
$$H^{d_i}_\m(M_i) \cong H^{d_i}_\m(M_i/M_{i-1}),\,\,\, H^j_\m(M_i) \cong H^j_\m(M_{i-1}) \text{\, if } j<d_i,\text{\, and } H^j_\m(M_i)=0 \text{\, otherwise}.$$

\noindent
In particular, $H^j_\m(M) =0$ for all $j>d_r$, and $H^j_\m(M) = H^j_\m(M_i)$ for all $j\leq d_i$.  In conclusion we have shown  that
\[
H^j_\m(M) = \begin{cases} H^{d_i}_\m(M_i) = H^{d_i}_\m(M_i/M_{i-1}) & \text{ if } j=d_i \text{ for some } i, \\ 0 & \text{\,\,otherwise}.
\end{cases} 
\]

Part (2) is clear once we notice that $0 =M_0 \subsetneq M_1 \subsetneq \ldots \subsetneq M_i$ and $0=M_j/M_j \subsetneq M_{j+1}/M_j \subsetneq \ldots \subsetneq M_i/M_j$ are sCM filtrations.
\end{proof}
}

\begin{example}\label{regulare}
Suppose that $x$ is an $M$-regular element and that $M/xM$ is sequentially Cohen-Macaulay; it is not true in general that $M$ is sequentially Cohen-Macaulay, as observed after \cite[Proposition 2.2]{TPDA17} - the statement of \cite[Theorem 4.7]{Sch99} is not correct. As a counterexample one can take a $2$-dimensional not Cohen-Macaulay local domain $R$ of depth 1, which is not sequentially Cohen-Macaulay by Example \ref{example sCM} (3). For instance, take  $R=\mathbb{Q}[|a^4,a^3b,ab^3,b^4|] \cong \mathbb{Q}[|z_1,z_2,z_3,z_4]/(z_2 z_3 - z_1 z_4, z_3^2 - z_2 z_4^2, z_1 z_3^2 - z_2^2 z_4, z_2^3 - z_1^2 z_3)$.  On the contrary,  every $0\neq x\in R$ is regular, $R/(x)$ is a 1-dimensional $R$-module and, therefore, always sequentially Cohen-Macaulay by Example \ref{example sCM} (2).
\end{example}

In the rest of this section, let $S=k[x_1,\ldots,x_n]$, with the standard grading.
Given a monomial ideal $I\subset S$, we let $G(I)$ denote the monomial minimal system of generators of $I$ and $m(I)=\max\{i \mid x_i \ {\rm divides } \ u \ {\rm for \ some } \ u \in G(I)\}$.

An important class of sequentially Cohen-Macaulay modules is given by rings defined by weakly stable ideals.
 
\begin{definition}
A monomial ideal $I\subseteq S$ is said to be {\it weakly stable} if for all monomials $u \in I$ and all integers $i,j$ with $1 \leq j < i \leq n$, there exists $t \in \mathbb{N}$ such that $ x_j^t u/x_i^\ell \in I$, where $\ell$ is the largest integer such that $x_i^\ell$ divides $u$.
\end{definition}

Observe that the condition of the previous definition is verified if and only if it is verified for all $u\in G(I)$.
In the  literature weakly stable ideals are also called ideals of Borel type, quasi-stable ideals, or ideals of nested type, see \cite{C04}, \cite{HH11}, \cite{BG06}. It is easy too see from the definition that stable, strongly stable, and $p$-Borel ideals are weakly stable. In particular, no matter what the characteristic of the field is, Borel fixed ideals are weakly stable, see also \cite[Theorem 4.2.10]{HH11} and, thus, generic initial ideals are always weakly stable.

Observe  that the saturation $I^{\rm sat}$ of a weakly stable ideal $I$ is $I:\m^{\infty}= I:x_n^\infty$, see \cite[Proposition 4.2.9]{HH11}; thus,   $I^{\rm sat}$ is  again weakly stable and $x_n$ does not divide any $u\in G(I^{\rm sat})$.

\begin{proposition}\label{wsscm}
Let $I$ be a weakly stable ideal of  $S=k[x_1, \dots, x_n]$. Then, $S/I$ is sequentially Cohen-Macaulay. 
\end{proposition}

\begin{proof}
    We prove the statement by induction on $m(I)$. If $m(I)=1$, then $I=(x_1^{r})$ for some positive integer $r$ and $S/I$ is Cohen-Macaulay. Assume now that $S/J$ is sequentially Cohen-Macaulay for every weakly stable ideal $J$ for which $m(J)<m(I)$.

    Let $S'=k[x_1, \dots, x_{m(I)}]$  and  $I'=I\cap S'$. Since $S/I \cong  S'/I' \otimes_{k}k[x_{m(I)+1},\dots, x_n]$ and a sCM filtration of $S'/I'$ is easily extended into one of $S/I$, it is sufficient to prove that $S'/I'$ is sequentially Cohen-Macaulay. Therefore, without loss of generality,  we may assume that $m(I)=n$, and that   $S/I^{\rm sat}$ is sequentially Cohen-Macaulay. Now,  $I^{\rm sat}/I$ is a non-trivial Artinian module, and the first non-zero module of a sCM filtration of $S/I^{\rm sat}$ has positive dimension, see Proposition \ref{sullecoom}. We may thus conclude that $S/I$ is sequentially Cohen-Macaulay, cf. Example \ref{example sCM} (4).
\end{proof}

\begin{remark}\label{bravog}  By \cite[Proposition 3.2]{BG06}, see also \cite[Proposition 4.2.9]{HH11} and \cite[Chapter 4]{C04}) a monomial ideal $I$ is weakly stable if and only if all its associated primes are generated by initial segments of variables, i.e. are of type $(x_1,\ldots,x_i)$ for some $i$. Since being sequentially Cohen-Macaulay is independent of coordinates changes on $S$, the above proposition shows that whenever the associated primes of a monomial ideal $I\subseteq S$ are totally ordered by inclusion, then $S/I$ is sequentially Cohen-Macaulay.
\end{remark}

\begin{example}
  Let $I=(x_1^4,x_1^2 x_2^2,x_1^3 x_3,x_1^2 x_2 x_3,x_1^2 x_2 x_4)\subseteq S=k[x_1,x_2,x_3,x_4]$. It is easy to verify that $I$ is a weakly stable ideal which is not strongly stable, by checking the definition or by computing its associated primes, by using the previous remark; thus, $S/I$ is sequentially Cohen-Macaulay by Proposition \ref{wsscm}.  We can construct an explicit sCM filtration proceeding as in its proof.  We have $m(I)=4$, and  $I^{\rm sat}=I:x_4^{\infty}=(x_1^4,x_1^2 x_2,x_1^3 x_3)$. This means that the first non-zero module of our filtration will be $(x_1^4,x_1^2 x_2,x_1^3 x_3)/I$. Now we consider  $I'=(x_1^4,x_1^2 x_2,x_1^3 x_3)$ as an ideal of $k[x_1,x_2,x_3]$ and compute  its saturation  $I''=I:x_3^{\infty}=(x_1^3,x_1^2 x_2)$.
    Hence, the second non-zero module of the filtration is $(x_1^3,x_1^2 x_2)/I$. Proceeding  in this way we obtain the filtration $    0 \subseteq (x_1^4,x_1^2 x_2,x_1^3 x_3)/I \subseteq (x_1^3,x_1^2 x_2)/I \subseteq (x_1^2)/I \subseteq (1)/I=S/I$, and it is easily seen that it is a sCM filtration of $S/I$.

\end{example}

\begin{example}\label{Alexander}
For the case  $M=R=S/I$, when $I=I_\Delta$ is the Stanley-Reisner ideal of a simplicial complex $\Delta$, there is a beautiful characterization of sequential Cohen-Macaulayness due to Duval, \cite[Theorem 3.3]{D96}.  Given a simplicial complex $\Delta$, let $\Delta(i)$ be the pure $i$-th skeleton of $\Delta$, i.e. the pure subcomplex of $\Delta$ whose facets are the faces of $\Delta$ of dimension $i$. Then, $S/I_{\Delta}$ is sequentially Cohen-Macaulay if and only if $S/I_{\Delta(i)}$ is Cohen-Macaulay for all $i$. 
Another important result in this context is that $I_\Delta$ is componentwise linear if and only if the Stanley-Reisner ring of its Alexander dual $\Delta^*$  is sequentially Cohen-Macaulay, see \cite[Theorem 2.1]{HH99} and \cite[Theorem 8.2.20]{HH11}. Moreover, it is known that if $\Delta$ is (nonpure) shellable, then $S/I_{\Delta}$ is sequentially Cohen-Macaulay, see \cite[Corollary 8.2.19]{HH11}.   
\end{example}

\begin{example}\label{prettybaby} Another class of examples of sequentially Cohen-Macaulay modules is given by pretty clean modules, which have been introduced by Herzog and Popescu in \cite{HP06} in order to characterize shellability of multicomplexes. A {\em pretty clean} module $M$ is a module that admits a {\em pretty clean} filtration, i.e. a prime filtration of $M$ by submodules $0=M_0 \subsetneq M_1 \subsetneq \ldots\subsetneq M_s=M$ such that each quotient $M_i/M_{i-1}$ is isomorphic to $S/\pp_i$, for some prime ideals $\pp_i$, with the following property: if $\pp_i\subsetneq \pp_j$ then $i>j$. For example, if $M=S/I$ with $I$ weakly stable, cf. \cite[Proposition 5.2]{HP06}, or is such that $\Ass(M)$ is a totally ordered set, cf. \cite[Proposition 5.1]{HP06} or Remark \ref{bravog}, then $M$ is pretty clean.  
By \cite[Theorem 4.1]{HP06} pretty clean modules are sequentially Cohen-Macaulay, provided every prime $\pp_i$ appearing in the pretty clean filtration is such that $S/\pp_i$ is Cohen-Macaulay.
  \end{example}
\normalcolor

\section{Characterizations of sequentially Cohen-Macaulay modules}

\noindent
The goal of this section is to present two characterizations of sequentially Cohen-Macaulay modules, due to Schenzel, see Theorem \ref{thm sCM and CMf} and Peskine, see Theorem \ref{peskine}.  As in the previous section, we let $(R,\m,k)$ be either a Noetherian local ring or a standard graded $k$-algebra with maximal homogeneous ideal $\m$;  we let $\dim(R)=n$.  In the second case, modules will be graded and homomorphisms homogeneous of degree 0.

\subsection{ Schenzel's characterization} Our main reference here is \cite{Sch99}.
By convention, the dimension of the zero module is set to be $-1$. Let $M$ be a finitely generated (graded) $R$-module of dimension $d$. Since $R$ is Noetherian,  for all $i=0,\ldots,d$, we can consider the largest (graded) submodule of $M$ of dimension $\leq i$,  and denote it by $\delta_i(M)$. By maximality, in this way one obtains a filtration $\mathcal{M}: 0 \subseteq \delta_0(M) \subseteq \delta_1(M) \subseteq \ldots \subseteq \delta_d(M)=M$, called the {\em dimension filtration of $M$}. Evidently, such a filtration is unique. 

   Given a set $X$ of prime ideals of $R$, we denote by $X_i=\{\pp \in X \,:\, \dim R/\pp=i\}$. Similarly, we define $X_{\leq i}$ and $X_{>i}$. 

\begin{remark} \label{remark submodule dimension and Ass}
Observe that $M$ has a non-zero submodule of dimension $i$ if and only if $\Ass(M)_i \ne \emptyset$. In fact, if $\pp \in \Ass(M)_i$, then $R/\pp$ is a non-zero submodule of $M$ of dimension $i$. Conversely, if $N$ is a non-zero submodule of $M$ of dimension $i$, then there must exist a minimal prime $\pp$ of $N$ such that $\dim(R/\pp) = \dim(N) = i$. Since $N \subseteq M$, we must also have that $\pp \in \Ass(M)$, as desired. In particular, we have that $\delta_i(M) = 0$ if and only if $\Ass(M)_{\leq i} = \emptyset$.
\end{remark}

Notice that $\delta_0(M)=H^0_\mm(M)$; the other modules in $\mathcal{M}$ can be described similarly, with the help of the minimal primary decomposition of $0$ as a submodule of $M$. In fact, let $\Ass(M)=\{\pp_1,\ldots,\pp_m\}$ and, for all  $i=0,\ldots,d$,  consider the set $\Ass(M)_{\leq i}$. We set

\[
\aa_i = \begin{cases} \prod\limits_{\pp\in \Ass(M)_{\leq i}} \pp & \text{ if } \Ass(M)_{\leq i} \ne\emptyset,  \\ R & \text{ otherwise.}
\end{cases}
\]

Now let us consider an irredundant primary decomposition $0 = \bigcap_{j=1}^m N_j$ inside $M$, where each $M/N_j$ is a $\pp_j$-primary module. 

\begin{proposition}[{\cite[Proposition 2.2]{Sch99}}]\label{sche02}
  Let $M$ be a finitely generated $R$-module of dimension $d$. With the above notation, for all $i=0,\ldots,d$ we have that
 \[
 \delta_i(M) = H^0_{\aa_i}(M)=\bigcap\limits_{\{j \ \mid \ \pp_j\in\Ass(M)_{> i}\}} N_j,
 \]
 where we let the intersection over the empty set  be equal to $M$.
\end{proposition}
\begin{proof}
We start with the first equality and assume that $\delta_i(M)=0$. By Remark \ref{remark submodule dimension and Ass} we have that $\Ass(M)_{\leq i} = \emptyset$ and in this case  $\aa_i = R$. Thus,  $H^0_{\aa_i}(M) = H^0_R(M) = 0 = \delta_i(M)$.

\noindent
Now assume that $\delta_i(M) \ne 0$ and observe that,  every associated prime of $\delta_i(M)$ has dimension $\leq i$ and, therefore, $\Ass(\delta_i(M)) \subseteq \Ass(M)_{\leq i}$. Thus, there is a power of $\aa_i$ which annihilates $\delta_i(M)$, and we get that $\delta_i(M) \subseteq H^0_{\aa_i}(M)$. On the other hand, there is a power of $\aa_i$ which annihilates $H^0_{\aa_i}(M)$ and, thus, $\dim(H^0_{\aa_i}(M)) \leq \dim(R/\aa_i) \leq i$. From the maximality of $\delta_i(M)$ it follows that $H^0_{\aa_i}(M) \subseteq \delta_i(M)$, as desired. 

\smallskip
For the second equality, observe that $N_j:_M \pp_j^\infty =M$ and  $N_j:_M x^\infty =N_j$ for any $x \notin \pp_j$, with $j\in \{1,\ldots,m\}$, since $M/N_j$ is $\pp_j$-primary. Also notice that
\[
H^0_{\aa_i}(M) = 0:_M \aa_i^\infty = \left(\bigcap_{j=1}^m N_j\right):_M \aa_i^\infty = \bigcap_{j=1}^m (N_j:_M \aa_i^\infty).
\]
Now, if $\pp_j \in \Ass(M)_{\leq i}$, then 
$
M \supseteq N_j:_M \aa_i^\infty \supseteq N_j :_M \pp_j^\infty =M,
$
forcing equality everywhere. Next, assume that $\pp_j \notin \Ass(M)_{\leq i}$. Then, $\pp \not\subseteq \pp_j$ for every $\pp \in \Ass(M)_{\leq i}$ since, otherwise, we would have $\dim(R/\pp_j) \leq \dim(R/\pp) \leq i$ and, thus, $\pp_j \in \Ass(M)_{\leq i}$. In this case, by choosing  $x \in \aa_i$ and $x \notin \pp_j$, we have that $N_j \subseteq N_j:_M \aa_i^\infty \subseteq N_j:_M x^\infty = N_j,$
and equalities hold. Summing up, we conclude that
\[
H^0_{\aa_i}(M) = \bigcap_{\{j \ \mid \ \pp_j \notin \Ass(M)_{\leq i}\}} N_j = \bigcap_{\{j \ \mid \ \pp_j \in \Ass(M)_{>i}\}} N_j.  \qedhere
\]
\end{proof}

\noindent
Note that this is consistent with our convention that  the intersection over the empty set is equal to $M$; in fact, for $i=d$ we have that $\Ass(M)_{\leq d} =\Ass(M)$ and $\Ass(M)_{>d} = \emptyset$. This agrees with the fact that $H^0_{\aa_d}(M) = H^0_{\sqrt{0}}(M) = M$.

 \begin{proposition}[\cite{Sch99}, Corollary 2.3]\label{assass} Let $M$ be a finitely generated $R$-module of dimension $d$; then, for all $i=0,\ldots,d$,  
    \begin{enumerate}[(1)]
    \item $\Ass(\delta_i(M))=\Ass(M)_{\leq i}$;
    \item $\Ass(M/\delta_i(M))=\Ass(M)_{>i}$;
    \item $\Ass(\delta_i(M)/\delta_{i-1}(M))=\Ass(M)_i$.
    \end{enumerate}
    
\end{proposition}
 \begin{proof} It is well known that,   $\Ass(M)=\Ass(H^0_\aa(M))\sqcup \Ass(M/H^0_\aa(M))$, for any ideal $\aa$ of $R$. Also notice that $\Ass(H^0_{\aa}(M))=\Ass(M)\cap V(\aa)$, for all $\aa$.

   \smallskip
   \noindent
Since $\delta_i(M)= H^0_{\aa_i}(M)$ by Proposition \ref{sche02}, the first equality descends from  Remark \ref{remark submodule dimension and Ass} and the above observation.
    
For (2), consider the short exact sequence $0 \lra \delta_i(M) \lra M \lra M/\delta_i(M) \lra 0$. By Proposition \ref{sche02}  one has that $\Ass(M)=\Ass(\delta_i(M))\sqcup \Ass(M/\delta_i(M))$, which is equal to $\Ass(M)_{\leq i}\sqcup\Ass(M)_{>i}$ by (1), and the second equality follows.
    
Finally, consider the short exact sequence  $0 \to \delta_{i-1}(M) \to \delta_i(M) \to \delta_i(M)/\delta_{i-1}(M)\to 0$ and observe that  $\Ass(\delta_i(M)/\delta_{i-1}(M))\subseteq  \Ass(M/\delta_{i-1}(M))$. We also have $\Ass(\delta_i(M)/\delta_{i-1}(M))$ $\subseteq \Ass(\delta_i(M))$ since $\delta_{i-1}(M)=H^0_{\aa_{i-1}}(M)=H^0_{\aa_{i-1}}(\delta_i(M))$. Thus, by (1) and (2) we have necessarily that $\Ass(\delta_i(M)/\delta_{i-1}(M))$ $\subseteq \Ass(M)_i$. On the other hand, by (1), if $\pp\in\Ass(M)_i$ then $\pp\not\in\Ass(M_{i-1})$, and therefore $\pp\in\Ass(\delta_i(M)/\delta_{i-1}(M))$.  
   \end{proof}

\noindent As a corollary of Propositions \ref{sche02} and \ref{assass}, one can  obtain another characterization of the dimension filtration, cf. \cite [Proposition 1.1]{HP06}.
  
  \begin{corollary}\label{hepo1.1}
    Let $M$ be a $d$-dimensional $R$-module.  A filtration of $M$ by submodules $0 \subseteq M_0 \subseteq \ldots \subseteq M_d=M$ is the dimension filtration of $M$ if and only if  $\Ass(M_i/M_{i-1})=\Ass(M)_i$, for all $i$. 
  \end{corollary}
  
The following definition was introduced in \cite[Definition 4.1]{Sch99}.
  \begin{definition}
  Let $M$ be a $d$-dimensional finitely generated $R$-module, and $\mathcal{M}:0 \subseteq \delta_0(M) \subseteq \ldots \subseteq \delta_d(M) =M$ be its dimension filtration. Then, $M$ is called {\em Cohen-Macaulay filtered} if for all $i \in \{0,\ldots,d\}$ the module $\delta_i(M)/\delta_{i-1}(M)$ is either zero or Cohen-Macaulay.   
 \end{definition}

  \noindent
  Notice that from the definition it immediately follows that if $\delta_i(M)/\delta_{i-1}(M)\neq 0$, then it has dimension $i$.
  
 A Cohen-Macaulay filtered module $M$ is sequentially Cohen-Macaulay. In fact, we may let $i_1$ denote  the smallest integer such that $\delta_{i_1}(M) \ne 0$ and set $M_1=\delta_{i_1}(M)$. Then, if $i_j$ is the smallest integer such that $\delta_{i_j-1}(M)\subsetneq \delta_{i_j}(M)$, we let $M_j=\delta_{i_j}(M)$. The resulting filtration $0 =M_0 \subsetneq M_1 \subsetneq \ldots \subsetneq M_r =M$ is clearly a sCM filtration of $M$.

A filtration $0=C_{-1}\subseteq C_0\subseteq C_1\subseteq \ldots \subseteq C_d=M$ such that each of its quotients $C_i/C_{i-1}$ is either zero or Cohen-Macaulay of dimension $i$ is called a {\em CM filtration} of $M$.  To see that the notions of sequentially Cohen-Macaulay and Cohen-Macaulay filtered modules coincide, we need the following result, which shows that if a CM filtration of a module exists, then it is unique and it is equal to its dimension filtration.

  \begin{proposition}[\cite{Sch99}, Proposition 4.3]\label{CMFvsDF}
Let $M$ be a $d$-dimensional finitely generated $R$-module. If $M$ has a CM filtration $0=C_{-1}\subseteq C_0\subseteq C_1\subseteq \ldots \subseteq C_d=M$, then $C_i = \delta_i(M)$ for every $i \in \{0,\ldots,d\}$.
  \end{proposition}
  \begin{proof}
We proceed by induction on $d$. If $d=0$, then $C_0=M=\delta_0(M)$ and this case is  complete. Assume henceforth that $d>0$, so that by induction we have that $C_i=\delta_i(C_{d-1})$  for all $i<d$. Since for $i<d$ we have that $\delta_i(M) = \delta_i(\delta_{d-1}(M))$, it suffices to show that $\delta_{d-1}(M) = C_{d-1}$.

Observe that $\Ass(C_0)\subseteq \Ass(M)_{\leq 0}$ and, for all $\pp\in\Ass(C_i/C_{i-1})$, we have $\dim A/\pp=i$ if $C_i\neq C_{i-1}$. With this information, we can inductively show that $\dim C_i\leq i$ and, accordingly  $C_i\subseteq \delta_i(M)$ for all $i$; in particular, $C_{d-1} \subseteq \delta_{d-1}(M)$. If $\Ass(M)_{\leq d-1} =\emptyset$, then $\delta_{d-1}(M) = 0$ by Remark \ref{remark submodule dimension and Ass}, and the desired equality is trivial. Otherwise,  we let
  $\aa=\aa_{d-1}=\prod\limits_{\pp\in \Ass(M)_{\leq d-1}} \pp$ and we claim that $H^0_\aa(M/C_{d-1}) = 0$. This is obvious if  $M/C_{d-1} =0$. Thus, assume $C_{d-1}\neq M$ and observe that $\aa$ contains a regular element of $M/C_{d-1}$. To see the latter, assume that $\aa \subseteq \bigcup\limits_{\pp\in\Ass(M/C_{d-1})}\pp$; then,  by prime avoidance we can find $\pp\in \Ass(M/C_{d-1})$ such that $\aa\subseteq\pp$ and, therefore,  $\pp$ contains a prime $\pp' \in \Ass(M)_{\leq d-1}$. This is a contradiction, since it would imply $d=\dim(M/C_{d-1}) = \dim(R/\pp) \leq \dim(R/\pp') \leq d-1$.
Finally, consider the short exact sequence $0\lra C_{d-1} \lra M \lra M/C_{d-1} \lra 0$. By Proposition \ref{sche02} we have that $\delta_{d-1}(M) = H^0_{\aa}(M)  = H^0_{\aa}(C_{d-1}) \subseteq C_{d-1}$. Thus, $C_{d-1} = \delta_{d-1}(M)$, and we are done.
  \end{proof}
  
\begin{theorem}[Schenzel] \label{thm sCM and CMf}
Let $M$ be a finitely generated $R$-module; then  $M$ is sequentially Cohen-Macaulay if and only if $M$ is Cohen-Macaulay filtered. Moreover, if $M$ is sequentially Cohen-Macaulay, then its sCM filtration is unique.
\end{theorem}

\begin{proof}
 We only need to show that the ``only if'' part.
   Let $\mathcal{M}: 0 = M_0 \subsetneq M_1 \subsetneq \ldots \subsetneq M_r=M$ be any sCM filtration of $M$, with $d=\dim(M)$, $d_0=-1$ and $d_i = \dim(M_i/M_{i-1})$. For $j \in \{-1, 0,\ldots, d\}$,  we also let $i(j)$ be the largest integer $i$ such that $d_{i} \leq j$, and we set $C_j=M_{i(j)}$. In this way we have constructed a filtration  $\mathcal{C}:0=C_{-1} \subseteq C_0 \subseteq \ldots \subseteq C_d=M$, which  is a CM filtration of $M$. By Proposition \ref{CMFvsDF} we may conclude that $C_i = \delta_i(M)$, and that $\mathcal{C}$ is the unique dimension filtration of $M$. Thus,  $M$ is Cohen-Macaulay filtered.

Observe that  one can reconstruct $\mathcal{M}$ from  $\mathcal{C}$; it follows that the modules $M_i$ only depend on $M$ as well, and $\mathcal{M}$ is also unique.
\end{proof}

\noindent Notice that, by the above theorem and  Proposition \ref{assass}, if $M$ is sequentially Cohen-Macaulay with sCM filtration $\mathcal{M}$, then $\Ass(M)=\sqcup_i \Ass(M_i/M_{i-1})$.  The comparison between sCM filtrations and dimension filtrations has also the following useful consequences. 

\begin{corollary} \label{coroll H0}
Let $M$ be a $d$-dimensional finitely generated $R$-module. The following are equivalent:
\begin{enumerate}[(1)]
\item $M$ is sequentially Cohen-Macaulay;
\item The modules $\delta_i(M)$ and $M/\delta_i(M)$ are sequentially Cohen-Macaulay for all $i \in \{0,\ldots,d\}$;
\item There exists $i \in \{0,\ldots,d\}$ such that $\delta_i(M)$ and $M/\delta_i(M)$ are sequentially Cohen-Macaulay.
\end{enumerate}

In particular,  $M$ is sequentially Cohen-Macaulay if and only if $M/H^0_\mm(M)$ is sequentially Cohen-Macaulay. 
\end{corollary}
\begin{proof}
Observe that, given  any $i \in \{0,\ldots,d\}$,  we have that $\delta_j(M) = \delta_j(\delta_i(M))$ for every $j \leq i$, and $\delta_j(M/\delta_i(M)) = \delta_j(M)/\delta_i(M)$ for every $j \geq i$. It follows that $0 \subseteq \delta_0(M) \subseteq \ldots \subseteq \delta_i(M)$ and $0 = \delta_i(M)/\delta_i(M) \subseteq \delta_{i+1}(M)/\delta_i(M) \subseteq \ldots \subseteq \delta_d(M)/\delta_i(M)$ are the dimension filtrations of $\delta_i(M)$ and $M/\delta_i(M)$ respectively. This yields all the assertions at once.
\end{proof}

Recall now that a finitely generated $R$-module $M$ is {\em unmixed} if $\dim(R/\pp) = \dim(M)$ for all $\pp \in \Ass(M)$.

\begin{corollary} Let $M$ be a finitely generated unmixed $R$-module. Then, $M$ is sequentially Cohen-Macaulay if and only if $M$ is Cohen-Macaulay.
\end{corollary}
\begin{proof}
One direction is clear. Let $d=\dim(M)$; in view of Remark \ref{remark submodule dimension and Ass}, the fact that $M$ is unmixed guarantees that $\delta_i(M) = 0$ for all $i<d$. By Theorem \ref{thm sCM and CMf}, $M$ is Cohen-Macaulay filtered, and thus $\delta_d(M)/\delta_{d-1}(M) = M$ is Cohen-Macaulay.
\end{proof}

\subsection{Peskine's characterization} 
We will always assume that $R$ has a canonical module $\omega_R$. In our setup, this assumption is not too restrictive: it is always the case when $R$ is standard graded, and it is true for instance if $R$ is complete local, cf. Remark \ref{perdopo}.

\begin{proposition}\label{fundaext}
 Let $R$ be an $n$-dimensional Cohen-Macaulay ring with canonical module $\omega_R$ and $M$ be sequentially Cohen-Macaulay, with sCM filtration $0=M_0\subsetneq M_1\subsetneq \ldots \subsetneq M_r=M$; also let $d_i=\dim M_i/M_{i-1}$, for $i=1,\ldots,r$. Then
\begin{enumerate}[(1)]
\item  for all $i=1,\ldots,r$, one has that $\Ext_R^{n-d_i}(M, \omega_R)\simeq \Ext_R^{n-d_i}(M_i/M_{i-1},\omega_R)$ is Cohen-Macaulay and has dimension $d_i$;
  \item $\Ext_R^{n-j}(M,\omega_R)=0$ whenever $j\not\in\{d_1,\ldots,d_r\}$;
\item $\Ext_R^{n-d_i}(\Ext_R^{n-d_i}(M,\omega_R),\omega_R)\simeq M_i/M_{i-1}$ for $i=1,\ldots,r$.
\end{enumerate}
\end{proposition}  

\begin{proof}
In the graded case,  (1) and (2) follow immediately from Proposition \ref{sullecoom}, graded local duality \cite[Theorem 3.6.19]{BH93} and \cite[Theorem 3.3.10 (c) (i)]{BH93}. In the local case, local duality yields (1) and (2) for the completion $\widehat{M}$ of $M$ as an $\widehat{R}$-module, see Example \ref{nuovo} (2). Since $\omega_{\widehat{R}} \cong \widehat{\omega_R} \cong \omega_R \otimes_R \widehat{R}$ and $\Ext^i_{\widehat{R}}(\widehat{N},\omega_{\widehat{R}}) \cong \Ext^i_R(N,\omega_R) \otimes_R \widehat{R}$ for any finitely generated $R$-module $N$, we conclude by faithful flatness of $\widehat{R}$ that (1) and (2) hold also in the local case.

Finally, by \cite[Theorem 3.3.10 (c) (iii)]{BH93} we have that $N \cong \Ext_R^{n-d}(\Ext_R^{n-d}(N,\omega_R),\omega_R)$ for any Cohen-Macaulay module $N$ of dimension $d$ and, thus,  the last statement follows immediately from  (1).
\end{proof}

\begin{remark}\label{MmoduloM_1} Let $M$ be as in the above proposition.
  \begin{enumerate}[(1)]
\item    By Parts (2) and (3) of the previous proposition, we have an isomorphism $$M_1\cong\Ext_R^{n-t}(\Ext_R^{n-t}(M,\omega_R),\omega_R),$$ where $t=\depth M$ and $M_1$ is $t$-dimensional and Cohen-Macaulay. We will show an extension of this fact in Lemma \ref{monomorphism}.

\item Notice that if $M$ is sequentially Cohen-Macaulay with  $\depth(M)=0$, then in particular $M_1= H^0_\mm(M)$. In fact, using \cite[Theorem 3.3.10 (c) (iii)]{BH93} and the short exact sequence $0 \lra H^0_\m(M) \lra M \lra M/H^0_\m(M) \lra 0$ we get that
\[
M_1 \cong \Ext_R^n(\Ext_R^n(M,\omega_R),\omega_R) \cong \Ext_R^n(\Ext_R^n(H^0_\m(M),\omega_R),\omega_R) \cong H^0_\mm(M).
\]
  \end{enumerate}
\end{remark}

We  recall now  the following crucial lemma, see \cite[Lemma 1.5]{HS01}.

\begin{lemma}\label{monomorphism}
  Let $R$ be a Cohen-Macaulay $n$-dimensional ring with canonical module $\omega_R$ and $M$ be a finitely generated $R$-module.

  Let also $\depth(M)=t$ and assume that $\Ext_R^{n-t}(M,\omega_R)$ is Cohen-Macaulay of dimension $t$. Then, there is a natural monomorphism $\alpha\,:\,\Ext_R^{n-t}(\Ext_R^{n-t}(M,\omega_R),\omega_R)\lra M$ such that 
 $$\Ext^{n-t}(\alpha)\,:\, \Ext_R^{n-t}(M,\omega_R)\lra \Ext_R^{n-t}(\Ext_R^{n-t}(\Ext_R^{n-t}(M,\omega_R),\omega_R)\omega_R)$$ is an isomorphism.
 \end{lemma}

 \begin{proof}
   We only give a proof in the graded case; the local one is handled similarly. Without loss of generality we may assume that $R$ is a polynomial ring, as we show next. We write $R=S/I$, where $S$ is a standard graded polynomial ring of dimension $m$ over a field $k$ and $I \subseteq S$ is homogeneous. Let $\mm$ and $\nn$ denote the graded maximal ideals of  $R$ and $S$ respectively. By graded Local Duality \cite[Theorem 3.6.19]{BH93} we have that  $\Ext^{m-i}_S(M,\omega_S)\cong \Hom_S(H^{i}_\nn(M),E_S(k))$ and   $\Ext^{n-i}_R(M,\omega_R)\cong \Hom_R(H^{i}_\mm(M),E_R(k))$. By Base Independence  $H^i_{\nn}(M)\cong H^i_\mm(M)$ for all $i$. Thus, by  $\Hom-\otimes$ adjointness and \cite[Lemma 3.1.6]{BH93} we obtain
$$\Hom_S(H^i_\nn(M),E_S(k)) \cong \Hom_R(H^i_\mm(M),\Hom_S(R, E_S(k)))
\cong \Hom_R(H^{i}_\mm(M),E_R(k)).$$
In particular, we have shown that $\Ext^{n-t}_R(M,\omega_R)\cong\Ext^{m-t}_S(M, \omega_S)$ and that  we may assume that $R$ is a polynomial ring.

\smallskip
Let now $
F_\bullet\: 0\lra F_{n-t}\lra \ldots\lra F_1\lra F_0 \lra 0$
be the minimal graded free resolution of $M$; applying the functor $\Hom_R(-,\omega_R)$ we obtain the dual complex 
\[
F^*_\bullet\,:\,\,\,\,\,\,\,0\lra F^*_{n-t}\lra \ldots \lra F^*_1\lra F^*_0\lra 0, 
\]
with  $H_0(F^*)=\Ext_R^{n-t}(M,\omega_R)$. Observe that $F^*_\bullet$ is a complex of free modules, since  $\omega_R \cong R(-n)$.
If we also let  $G_\bullet$ denote  the minimal graded free resolution of $\Ext_R^{n-t}(M,\omega_R)$, then there is a map of complexes $\phi_\bullet\,:\, F^*_\bullet\lra G_\bullet$ which lifts the identity map of   
$\Ext_R^{n-t}(M,\omega_R)$. Since the last one is Cohen-Macaulay of dimension $t$, the length of  $G_\bullet$ is the same as  
that of  $F^*_\bullet$, namely $n-t$. Applying the functor $\Hom_R(-,\omega_R)$ to $\phi_\bullet$, we obtain a map of complexes $\phi^*_\bullet\: G^*_\bullet\lra F_\bullet^{**}$ which, in turn, gives a map on the zero-th cohomology:
\[
\alpha=H_0(\phi^*_\bullet)\,:\,  
\Ext_R^{n-t}(\Ext_R^{n-t}(M,\omega_R),\omega_R)=H_0(G^*_\bullet)\to H_0(F_\bullet^{**}) \cong H_0(F_\bullet)=M.
\]

\noindent Applying again the functor $\Hom_R(-,\omega_R)$, this time to $\phi^*_\bullet$, we obtain  a map of complexes $\phi^{**}_\bullet: F_\bullet^{***} \to G^{**}_\bullet$ and, thus, a map 
\[
\begin{split}
H_0(\phi^{**}_\bullet): H_0(F^{***}_\bullet) \cong H^0(F^*_\bullet)& \cong \Ext^{n-t}_R(M,\omega_R) \\
& \lra \Ext^{n-t}_R(\Ext^{n-t}_R(\Ext^{n-t}_R(M,\omega_R),\omega_R),\omega_R) = H_0(G^{**}_\bullet),
\end{split}
\]
which coincides with the map $\Ext^{n-t}(\alpha)$. The canonical isomorphism $\psi: G_\bullet^{**} {\lra} G_\bullet$, together with the fact that $\Ext^{n-t}_R(M,\omega_R)$ is Cohen-Macaulay of dimension $t$, gives an isomorphism $H_0(\psi):\Ext^{n-t}_R(\Ext^{n-t}_R(\Ext^{n-t}_R(M,\omega_R),\omega_R),\omega_R) \lra \Ext^{n-t}_R(M,\omega_R)$, and one can verify that 
\[
H_0(\psi) \circ H_0(\phi_\bullet^{**}) = H_0(\phi_\bullet) = {\rm id}_{\Ext^{n-t}_R(M,\omega_R)}.
\]

Note that $H_0(\psi)$ is the inverse of the isomorphism of \cite[Theorem 3.3.10 (c) (iii)]{BH93}, and this implies that $H_0(\phi^{**}_\bullet) = \Ext^{n-t}(\alpha)$ is the natural isomorphism of the same theorem.

To conclude the proof, it remains to be shown that $\alpha$ is a monomorphism. 

\noindent
Let $N = \Ext^{n-t}_R(\Ext^{n-t}_R(M,\omega_R),\omega_R)$. We show that $\alpha$ is injective once we localize at every associated prime of $N$ and then we are done, since, if $\ker(\alpha) \ne 0$, its associated primes would be contained in those of $N$.

  \noindent Let $\pp \in \Ass(N)$; since $N$ is Cohen-Macaulay of dimension $t$, we have that $\dim(R/\pp) = t$  and  $\dim(R_\pp) = n-t$. By replacing $R$ with $R_\pp$, $M$ with $M_\pp$, and $\omega_R$ with $(\omega_R)_\pp \cong \omega_{R_\pp}$, the proof will be complete once we show that $\alpha$ is injective in the case $t=0$. To this end observe that, as in Remark \ref{MmoduloM_1}, the short exact sequence $0 \to H^0_\m(M) \to M \to M/H^0_\m(M) \to 0$ and \cite[Theorem 3.3.10 (c) (iii)]{BH93} yield 
\[
\Ext^n_R(\Ext^n_R(M,\omega_R),\omega_R) \cong \Ext^n_R(\Ext^n_R(H^0_\m(M),\omega_R),\omega_R) \cong H^0_\m(M),
\]
and $\alpha$ composed with this isomorphism becomes just the inclusion of $H^0_\m(M)$ inside $M$.
\end{proof}

The equivalence between the first two conditions in the next theorem was announced in \cite{St96} without a proof, but citing a spectral sequence argument due to Peskine. Here we give another proof of this fact, see also \cite[Theorem 5.5]{Sch99} where the equivalence with the third condition is proved.

\begin{theorem}[Peskine]\label{peskine}
Let $R$ be a Cohen-Macaulay ring of dimension $n$ with canonical module $\omega_R$, and $M$ be a finitely generated $d$-dimensional $R$-module. Then, the following are equivalent:
\begin{enumerate}[(1)]
\item $M$ is sequentially Cohen-Macaulay;
\item $\Ext_R^{n-i}(M,\omega_R)$ is either $0$ or Cohen-Macaulay of dimension $i$ for all $i \in \{0,\ldots,d\}$;
\item $\Ext_R^{n-i}(M,\omega_R)$ is either $0$ or Cohen-Macaulay of dimension $i$ for all $i \in \{1,\ldots,d-1\}$.
\end{enumerate}
\end{theorem}

\begin{proof} The implication (1) $\Rightarrow$ (2) follows at once by Proposition \ref{fundaext}, and clearly (2) implies (3). 

Now assume (3); we proceed by induction on $d-t$, where  $t=\depth(M)$. If $t=d$, then $M$ is Cohen-Macaulay, and hence sequentially Cohen-Macaulay. Assume that $t<d$. If $t=0$, then $0 \ne \Ext^n_R(M,\omega_R)$ has finite length, and hence it is Cohen-Macaulay. Either way, thanks to our assumption we have that $0 \ne \Ext^{n-t}_R(M,\omega_R)$ is $t$-dimensional and Cohen-Macaulay. By Lemma \ref{monomorphism}, there is an injective homomorphism $\alpha \,:\, \Ext_R^{n-t}(\Ext_R^{n-t}(M,\omega_R),\omega_R)\lra M$. Since the first module is $t$-dimensional Cohen-Macaulay by \cite[Theorem 3.3.10 (c)]{BH93}, then so is its image, say $M_1$, which is a submodule of $M$. It follows that $\depth (M_1)=\dim (M_1) =t=\depth(M)<d=\dim(M)$.

  Consider the short exact sequence $0\lra M_1\lra M\lra M/M_1\lra 0$ and  the induced sequence in cohomology obtained by applying the functor $\Hom_R(-,\omega_R)$. We then have isomorphisms $\Ext_R^j(M/M_1,\omega_R)\cong \Ext_R^j(M,\omega_R)$ for all $j\neq n-t,n-t+1$ and the exact sequence
\begin{equation*}
%  \begin{split}
    0\to \Ext_R^{n-t}(M/M_1,\omega_R)\to \Ext_R^{n-t}(M,\omega_R)%\
  \stackrel{\beta}{\to } \Ext_R^{n-t}(M_1,\omega_R)\to  \Ext_R^{n-t+1} (M/M_1,\omega_R)
   \to  0.
% \end{split}
%
\end{equation*}

\noindent
By Lemma \ref{monomorphism} we know that the map $\Ext^{n-t}(\alpha)$ is an isomorphism, and therefore $\beta$ is an isomorphism as well. It follows that $\Ext^j_R(M,\omega_R) \cong \Ext^j_R(M/M_1,\omega_R)$ for every $j \ne n-t$, while $\Ext^{n-t}_R(M/M_1,\omega_R)=0$.

\noindent
This shows in particular that $\depth(M/M_1) >t$; since $\dim(M/M_1) =d$,  we may apply induction and obtain that $M/M_1$ is sequentially Cohen-Macaulay. Let $0 = M_1/M_1 \subsetneq M_2/M_1 \subsetneq \ldots \subsetneq M_r/M_1 =M/M_1$ be a sCM filtration. Since $M_1$ is Cohen-Macaulay of dimension $t$ and $\depth(M_2/M_1) = \depth(M/M_1)>t$ by Example \ref{example sCM} (4) , we deduce that $0 = M_0 \subsetneq M_1 \subsetneq M_2 \subsetneq \ldots \subsetneq M_r=M$ is a sCM filtration and, thus $M$ is sequentially Cohen-Macaulay. 
\end{proof}

\begin{corollary}\label{sumviceversa} Let $R$ be Cohen-Macaulay with canonical module $\omega_R$, and  $M$, $N$ be finitely generated $R$-modules. Then, $M$ and $N$ are sequentially Cohen-Macaulay if and only if $M \oplus N$ is sequentially Cohen-Macaulay.
\end{corollary}
\begin{proof}
Let $n=\dim R$. We have already showed in Example \ref{nuovo} (1) that, if $M$ and $N$ are sequentially Cohen-Macaulay, then so is $M \oplus N$. This also follows immediately from Theorem \ref{peskine}, since if $\Ext^{n-i}_R(M,\omega_R)$ and $\Ext^{n-i}_R(N,\omega_R)$ is either zero or Cohen-Macaulay of dimension $i$, then so is $\Ext^{n-i}_R(M,\omega_R) \oplus \Ext^{n-i}_R(N,\omega_R)  \cong \Ext^{n-i}_R(M \oplus N,\omega_R)$. For the converse, it suffices to observe that if the direct sum of two modules is zero or Cohen-Macaulay of a given dimension, then so is each of its  summand.
\end{proof}

%\color{red}
\begin{remark}\label{loca}
We remark that sequential Cohen-Macaulayness behaves well with respect to localization; see for instance \cite[Proposition 4.7]{CN03} or \cite[Proposition 2.6]{CGT13}. For any sequentially Cohen-Macaulay $R$-module $M$ and $\pp\in \Supp(M)$ one has that $M_\pp$ is a sequentially Cohen-Macaulay $R_\pp$-module and, in fact, one can recover its dimension filtration from that of $M$. Let $s=\dim (R/\pp)$ and consider the quotients $\delta_i(M)/\delta_{i-1}(M)$ of the dimension filtration of $M$. If not zero, they are $i$-dimensional Cohen-Macaulay, and their localization is either zero or Cohen-Macaulay of dimension $i-s$. Now let $N_i=(\delta_{i+s}(M))_\pp$ for all $i \geq 0$ such that $i+s\leq d=\dim(M)$, i.e., for $i=0,\ldots,d-s$, and observe that $0\subseteq N_0 \subseteq \ldots \subseteq N_{d-s}=M_\pp$ is a CM filtration of $M_\pp$, with quotients $N_i/N_{i-1}\simeq (\delta_{i+s}(M)/\delta_{i+s-1}(M))_\pp$. If $R$ is Cohen-Macaulay, the fact that $M_\pp$ is sequentially Cohen-Macaulay for all $\pp \in \Supp(M)$ is also a consequence of Theorem \ref{peskine}. See \cite{TPDA18} for other results about localization and sequentially Cohen-Macaulay modules.
\end{remark}
%\normalcolor

We now recall some definitions needed to state the next result, and that we will use frequently in the next sections. Given an $R$-module $N$, we let $\Ass^\circ(N) = \Ass(N) \smallsetminus \{\m\}$.

\begin{definition}\label{filterstrictlyfilter} Let $M\neq 0$ be a finitely generated graded $R$-module.
\begin{enumerate}[(1)]
\item A homogeneous element $0 \ne y \in \m$ is {\em filter regular} for $M$ if $y \notin \bigcup_{\pp \in \Ass^\circ(M)} \pp$. 

  \noindent
  A sequence of homogeneous elements $y_1,\ldots,y_t \in \m$ is a {\em filter regular sequence} for $M$ if $y_{i+1}$ is a filter regular element for $M/(y_1,\ldots,y_i)M$ for all $i \in \{0,\ldots,t-1\}$.
\item A homogeneous element $0 \ne y \in \m$ is {\em strictly filter regular} for $M$ if $y \notin\bigcup_{\pp \in \Ass^\circ(X(M))} \pp$, where $X(M) = \bigoplus_{i \in \NN} \Ext^i_R(M,R)$.

    \noindent A sequence of homogeneous elements $y_1,\ldots,y_t \in \m$ is a {\em strictly filter regular sequence} for $M$ if $y_{i+1}$ is a strictly filter regular element for $M/(y_1,\ldots,y_i)M$ for all $i \in \{0,\ldots,t-1\}$.
\end{enumerate}
\end{definition}

\begin{remark} \label{remark SF -> F} Regular sequences are clearly filter regular sequences. Moreover, strictly filter regular sequences are filter regular by the graded version of \cite[Corollary 11.3.3]{BSxx}. When the field $k$ is infinite, by Prime Avoidance any sequence of general forms is strictly filter regular.
\end{remark}

We conclude this section with the following result, which clarifies how sequential Cohen-Macaulayness behaves with respect to quotients, cf. Example \ref{regulare}; it will also be useful later on.

\begin{proposition} \label{hscor1.9}
  Let $R$ be a Cohen-Macaulay ring of dimension $n$ with canonical module $\omega_R$; let $M$ be a $d$-dimensional finitely generated $R$-module, and $x \in R$ a strictly filter regular element for $M$. Then,
  \begin{enumerate}[(1)]
  \item  If $M$ is sequentially Cohen-Macaulay, then $M/xM$ is sequentially Cohen-Macaulay.
  \item The converse holds if $x$ is regular for all non-zero $\Ext^{n-i}_R(M,\omega_R)$ with $i>0$.
  \end{enumerate}
\end{proposition}
\begin{proof}
Assume that $M$ is sequentially Cohen-Macaulay.       Since $x$ is strictly filter regular, it is filter regular by Remark \ref{remark SF -> F},  $L=0:_M x$ has finite length and $\Ext^{n-i}_R(\overline{M},\omega_R) \cong \Ext^{n-i}_R(M,\omega_R)$ for all $i>0$, where $\overline{M}=M/L$.  Then, the long exact sequence obtained by applying the functor $\Hom_R(-,\omega_R)$ to the short exact sequence $0 \lra \overline{M} \stackrel{\cdot x}{\lra} M \lra M/xM \lra 0$ gives  a long exact sequence 

\begin{equation*}
  \begin{split}
\ldots& \stackrel{\cdot x}\lra \Ext^{n-i}_R(M,\omega_R) \lra \Ext^{n-(i-1)}_R(M/xM,\omega_R) \lra \Ext^{n-(i-1)}_R(M,\omega_R) \to \ldots \\
\ldots& \stackrel{\cdot x}{\lra}  \Ext^{n-1}_R(M,\omega_R) \lra  \Ext^n_R(M/xM,\omega_R) \lra \Ext^n_R(M,\omega_R) \lra 0.    
  \end{split}
\end{equation*}
By Theorem \ref{peskine} we have that      each non-zero $\Ext^{n-i}_R(M,\omega_R)$ is Cohen-Macaulay of dimension $i$, and in this case $x$ is $\Ext^{n-i}_R(M,\omega_R)$-regular when $i>0$.  For $i>1$ we then have short exact sequences 
\[
0 \lra \Ext^{n-i}_R(M,\omega_R) \stackrel{\cdot x}{\lra} \Ext^{n-i}_R(M,\omega_R) \lra \Ext^{n-(i-1)}_R(M/xM,\omega_R) \lra 0,
\]
and it follows that $\Ext^{n-(i-1)}_R(M/xM,\omega_R) \cong \Ext^{n-i}_R(M,\omega_R) \otimes_R R/(x)$ is Cohen-Macaulay of dimension $i-1$ for all $i>1$. We conclude that $M/xM$ is sequentially Cohen-Macaulay using the implication (3) $\Rightarrow$ (1) of Theorem \ref{peskine}.

For the converse, the fact that $x$ is regular for all non-zero $\Ext^{n-i}_R(M,\omega_R)$ with $i>0$ shows that the above long exact sequence of $\Ext$ modules breaks into short exact sequences 
\[
0 \lra \Ext^{n-i}_R(M,\omega_R) \stackrel{\cdot x}{\lra} \Ext^{n-i}_R(M,\omega_R) \lra \Ext^{n-(i-1)}_R(M/xM,\omega_R) \lra 0.
\]
for all $i>1$. By Theorem \ref{peskine} we have that $\Ext^{n-(i-1)}_R(M/xM,\omega_R)$ is either zero or Cohen-Macaulay of dimension $i-1$, and thus $\Ext^{n-i}_R(M,\omega_R)$ is either zero or Cohen-Macaulay of dimension $i$ for all $i>1$. Since $x$ is assumed to be regular on $\Ext^{n-1}_R(M,\omega_R)$, and $\dim(\Ext^{n-1}_R(M,\omega_R)) \leq 1$, we have that $\Ext^{n-1}_R(M,\omega_R)$ is either zero, or Cohen-Macaulay of dimension $1$. It follows  again from the implication (3) $\Rightarrow$ (1) of Theorem \ref{peskine} that $M$ is sequentially Cohen-Macaulay.
\end{proof}

%\color{blue}
\begin{remark}\label{parsys}
  There are many other interesting results about  sequentially Cohen-Macaulay modules and their characterizations which do not find space in this note. For instance, 
  in \cite[Theorem 5.1]{CC07} it is proven that a module is sequentially Cohen-Macaulay if and only if each module of its dimension filtration is pseudo Cohen-Macaulay. In \cite[Theorem 1.1]{CL09}, the sequential Cohen-Macaulayness of $M$ is characterized in terms of the existence of one good system of parameters of $M$ which has the property of parametric decomposition; see also Theorems 3.9 and 4.2 in \cite{CC07} for other characterizations which involve good systems of parameters and $dd$-sequences.

Moreover, in \cite{CGT13} it is investigated how the sequential Cohen-Macaulay property behaves in relation to taking associated graded rings and Rees algebras, see also  \cite{TPDA17} for more results of this type.
\end{remark}

\section{Partially sequentially Cohen-Macaulay modules}
We are going to study next the notion of partially sequentially Cohen-Macaulay module, as introduced in \cite{SS17}, which naturally generalizes that of sequentially Cohen-Macaulay modules. Thanks to Schenzel's Theorem \ref{thm sCM and CMf}, the definition can be given in terms of the dimension filtration of the module. Throughout this section we let $R=k[x_1,\dots, x_n]$ be a standard graded polynomial ring over an infinite field $k$ with homogeneous maximal ideal $\mm=(x_1, \dots, x_n)$. Recall that, in this case, $R$ has a graded canonical module $\omega_R \cong R(-n)$. We consider finitely generated graded $R$-modules $M$; when $M=0$, we set $\depth(M)=+\infty$ and $\dim(M)=-1$, as usual.  We let $d=\dim(M)$.

\begin{definition} \label{Def i-sCM} Let $i\in\{0,\ldots,d\}$ and let $\{\delta_j(M)\}_j$ be the dimension filtration of $M$; $M$ is called {\it $i$-partially sequentially Cohen-Macaulay}, $i$-sCM for short, if $\delta_j(M)/\delta_{j-1}(M)$ is either zero or Cohen-Macaulay for all $i \leq j \leq d$. 
  \end{definition}

  \noindent
  By definition and Corollary \ref{coroll H0}, a module $M$ is $0$-sCM if and only if $M$ sequentially Cohen-Macaulay  if and only if $M$ is $1$-sCM.

  \begin{example}\label{Example i-sCM}
    \begin{enumerate}[(1)] 
\item Let $M$ be a sequentially Cohen-Macaulay module. The simplest way of constructing an $i$-sCM module which is not sequentially Cohen-Macaulay, is perhaps taking a non-sequentially Cohen-Macaulay module $N$ of dimension strictly smaller than $i$, and condider  their direct sum $M \oplus N$, cf. Example \ref{nuovo} (1).

  \item Let $M=R/I$, where $I=(x_1) \cap (x_2,x_3) \cap (x_1^2,x_4,x_5) \subset R=k[x_1,x_2,x_3,x_4,x_5]$. With the help of Proposition \ref{sche02}, we can construct the dimension filtration $0=\delta_{-1}=\delta_0 = \delta_1 \subseteq  \delta_2=((x_1) \cap (x_2,x_3))/I \subseteq \delta_3=(x_1)/I \subseteq \delta_4=R/I$ of $M$. Then,  $\delta_4/\delta_3$ and $\delta_3/\delta_2$ are Cohen-Macaulay of dimension $4$ and $3$ respectively, but $0\neq \delta_2/\delta_1$ is not Cohen-Macaulay. Hence, $M$ is an example of a $3$-sCM which is not  $2$-sCM.    
  \end{enumerate}

\end{example}

\begin{remark}\label{remark jsCM-0}
Observe that $M$ is $i$-sCM if and only if $M/\delta_{i-1}(M)$ is sequentially Cohen-Macaulay. This follows at once  recalling that the dimension filtration $\{\gamma_j\}_j$ of $M/\delta_{i-1}(M)$ is such that  $\gamma_j = \delta_j(M)/\delta_{i-1}(M)$ for $j \geq i$ and $\gamma_j=0$ otherwise. Notice that, since  $\gamma_j=0$ for all $j \leq i-1$, if $M$ is $i$-sCM then  $H^j_\m(M/\delta_{i-1}(M)) = 0$ for all $j \leq i-1$, by Proposition \ref{sullecoom}.
\end{remark}
%}
Given a graded free presentation of $M \cong F/U$, we  denote by $\{e_1, \ldots, e_r\}$  a graded basis of $F$. We consider $R$ together with the pure reverse lexicographic ordering $>$ such that $x_1 > \ldots > x_n $; recall that $>$ is not a monomial order on $R$, but by definition it agrees with the reverse lexicographic order that refines it on monomials of the same degree. 

  We extend $>$ to $F$ in the following way: given monomials $ue_i$ and $ve_j$ of $F$, set
 \begin{equation*}
 \begin{split} ue_i>ve_j \text{ if }& \big(\deg(ue_i)>\deg(ve_j)\big), \text{ or } \big(\deg(ue_i)=\deg(ve_j) \text{ and } u>v\big)\\
  \text{ or }& \big(\deg(ue_i)=\deg(ve_j),\,\,u=v \text{ and } i<j\big). 
 \end{split}
\end{equation*}
We shall consider this order until the end of the section, and denote by $\Gin(U)$ the generic initial module of $U$ with respect to $>$. Since the action of ${\rm GL}_n(k)$ on  $R$ as change of coordinates can be extended in an obvious way to $F$,\, $\Gin(U)$ simply results to be the initial submodule ${\rm in}_>(gU)$ where $g$ is a general change of coordinates.
 
We prove next some preliminary facts which are needed later on. Given a graded submodule $V \subseteq F$ with $\dim(F/V)=d$, for all $j\in \{-1,\ldots,d\}$ we denote by $V^{\langle j \rangle}$ the $R$-module such that  $V^{\langle j \rangle}/V=\delta_j(F/V)$. Several results contained in the next two lemmata can be found in \cite{Go16}, where they are proved in the ideal case.

\begin{lemma}\label{Afshin}
With the above notation,
\begin{enumerate}[(1)]
\item $\Gin(U^{\langle j \rangle}) \subseteq \Gin(U)^{\langle j \rangle}$;
\item $U^{\langle j \rangle}=(U^{\langle j \rangle})^{\langle j \rangle}$;
\item if $V$ is a graded submodule of $F$ such that  $U \subseteq V$, then $U^{\langle j \rangle} \subseteq V^{\langle j \rangle}$;
\item $\Gin(U^{\langle j \rangle})^{\langle j \rangle}=\Gin(U)^{\langle j \rangle}$.
\end{enumerate}
\end{lemma}

\begin{proof}
(1) Notice that  $\Gin(U^{\langle j \rangle})/\Gin(U)$ and $U^{\langle j \rangle}/U$ have the same Hilbert series, hence the same dimension, which is less than or equal to $j$. Since $\Gin(U)^{\langle j \rangle}/\Gin(U)=\delta_j(F/\Gin(U))$, we have the desired inclusion. 

\smallskip
\noindent
(2) Since $U \subseteq U^{\langle j \rangle}$, one inclusion is clear. Now consider the short exact sequence
\[
0 \lra U^{\langle j \rangle}/U \lra (U^{\langle j \rangle})^{\langle j \rangle}/U \lra (U^{\langle j \rangle})^{\langle j \rangle}/U^{\langle j \rangle} \lra 0.
\]
Since the dimensions of $U^{\langle j \rangle}/U$ and $(U^{\langle j \rangle})^{\langle j \rangle}/U^{\langle j \rangle}$ are less than or equal to $j$, it follows that also $\dim((U^{\langle j \rangle})^{\langle j \rangle}/U) \leq j$ and,  hence, $(U^{\langle j \rangle})^{\langle j \rangle} \subseteq U^{\langle j \rangle}$.

\smallskip
\noindent
(3) Since $U \subseteq U^{\langle j\rangle}\cap V$, we have that $\dim((U^{ \langle j \rangle}+V)/V) \leq \dim(U^{\langle j \rangle}/U) \leq j$, which implies $U^{ \langle j \rangle}\subseteq U^{\langle j \rangle}+V \subseteq V^{\langle j \rangle}$.

\smallskip
\noindent
(4)  Since $U \subseteq U^{\langle j \rangle}$, it immediately follows from (1) that $\Gin(U) \subseteq \Gin(U^{\langle j \rangle}) \subseteq \Gin(U)^{\langle j \rangle}$ and, accordingly, 
$\Gin(U)^{\langle j \rangle} \subseteq \Gin(U^{\langle j \rangle})^{\langle j \rangle}$. On the other hand, by Parts (1) and (2), the latter is contained in  $(\Gin(U)^{\langle j \rangle})^{\langle j \rangle}=\Gin(U)^{\langle j \rangle}$.
\end{proof}

We denote the Hilbert series of a graded $R$-module $N$ by $\Hilb(N)=\Hilb(N,z)$. We also let $h^j(N) = \Hilb(H^j_\m(N))$.

\begin{lemma}\label{lemma SbSt}
Let $M \cong F/U$, with dimension filtration $\{\delta_j\}_j$; then,  the following holds:
\begin{enumerate}[(1)]
\item $M$ is $i$-sCM if and only if $M/H^{0}_\mm(M)$ is $i$-sCM.
\item If $M$ is $i$-sCM, then $H^j_\mm (M) \cong H^j_\mm(\delta_j) \cong H^j_\mm (\delta_j/\delta_{j-1})$ for all $j \geq i$.
\item If $M$ is $i$-sCM, then $(z-1)^jh^j(M)=(1-z)^j\Hilb(\delta_j/\delta_{j-1})$ for all $j \geq i$.
\item Let $\{\gamma_j\}_j$ be the dimension filtration of $F/\Gin(U)$; if $M$ is $i$-sCM, then $\Hilb(\delta_j/\delta_{j-1})=\Hilb(\gamma_j/\gamma_{j-1})$ for all $j \geq i$. 
\end{enumerate}
\end{lemma}

  \begin{proof}
    Since $H^0_\mathfrak{m}(R)=\delta_0$, the dimension filtration of $M/H^0_\mathfrak{m}(R)$ is $\{\delta_j/\delta_0\}_j$, which shows the first part.    The long exact sequence in cohomology induced by $0 \to \delta_{j-1} \to \delta_j \to \delta_j/\delta_{j-1}\to 0$ and the Cohen-Macaulayness of $\delta_j/\delta_{j-1}$ for all $j\geq i$ easily imply (2).

    \smallskip
    \noindent
Let us fix $j\geq i$ and prove (3). If $j=0$ the assertion is clear; thus we may assume $j>0$  and by way of Part (2) that $\delta_j/\delta_{j-1}$ is $j$-dimensional Cohen-Macaulay. Let $x$ be a $(\delta_j/\delta_{j-1})$-regular element of degree one; then, the short exact sequence given by multiplication by $x$ induces a short exact sequence in cohomology  
$$0 \lra H_{\m}^{j-1}\left(\frac{\delta_{j}/\delta_{j-1}}{x(\delta_{j}/\delta_{j-1})}\right) \lra H^j_{\m}(\delta_{j}/\delta_{j-1})(-1) \lra H^j_{\m}(\delta_{j}/\delta_{j-1}) \lra 0$$ 
together with  (2) imply that $h^{j-1}((\delta_{j}/\delta_{j-1})/(x(\delta_{j}/\delta_{j-1}))=(z-1) \, h^j(\delta_{j}/\delta_{j-1})=(z-1) \, h^j(M)$.

Thus, one can easily prove that $h^{0}((\delta_{j}/\delta_{j-1})/({\bf x}(\delta_{j}/\delta_{j-1})))=(z-1)^j \, h^j(M)$, where  ${\bf x}$ is a $(\delta_j/\delta_{j-1})$-maximal regular sequence. Since $\dim((\delta_{j}/\delta_{j-1})/({\bf x}\delta_{j}/\delta_{j-1}))=0$, we also have $h^{0}((\delta_{j}/\delta_{j-1})/({\bf x}\delta_{j}/\delta_{j-1}))=\Hilb((\delta_{j}/\delta_{j-1})/({\bf x}\delta_{j}/\delta_{j-1})) = (1-z)^j\Hilb(\delta_{j}/\delta_{j-1})$, and the proof of (3) is complete.

\smallskip
\noindent
Finally, to prove (4) we  show that $\Hilb(U^{\langle j \rangle}/U^{\langle j-1 \rangle})=\Hilb(\Gin(U)^{\langle j \rangle}/(\Gin(U)^{\langle j-1 \rangle}))$ holds for all $j \geq i$. Actually, we prove more, i.e. that $\Hilb(U^{\langle j \rangle})=\Hilb(\Gin(U)^{\langle j \rangle})$ for all $j \geq i$; since $\Gin(U^{\langle j \rangle}) \subseteq \Gin(U)^{\langle j \rangle}$ by Lemma \ref{Afshin} (1) for all $j$,  the last equality is equivalent to proving that $\Gin(U^{\langle j \rangle}) = \Gin(U)^{\langle j \rangle}$ for all $j \geq i$, and this is what we do.

Consider now, for all $j$, the short exact sequences $0  \to U^{\langle j \rangle}/U^{\langle j-1 \rangle} \to F/U^{\langle j-1 \rangle} \to F/U^{\langle j \rangle} \to 0;$
we see inductively that $\depth(F/U^{\langle j \rangle}) \geq j+1$ for all $j \geq i$. For $j=d$ and if $U^{\langle j\rangle}/U^{\langle j-1\rangle}=0$ this is obvious; otherwise, since $M$ is $i$-sCM,  $U^{\langle j\rangle}/U^{\langle j-1\rangle}$ is $j$-dimensional Cohen-Macaulay for all $j\geq i$ and  by \cite[Proposition 1.2.9]{BH93} we get that
$\depth F/U^{\langle j-1\rangle}\geq \min\{ j, j+1\}=j$.

\noindent For all graded submodules $V\subseteq F$ it is well-known that $\depth(F/\Gin(V))=\depth(F/V)$ and that $F/\Gin(V)$ is sequentially Cohen-Macaulay. Thus, $j+1$ $\leq$ $\depth(F/\Gin(U^{\langle j \rangle}))$ by what we proved above, and  Proposition \ref{sullecoom} together with Theorem \ref{thm sCM and CMf} imply that the latter is also  equal to the smallest integer $t$ such that $\Gin(U^{\langle j \rangle}) \subsetneq \Gin(U^{\langle j \rangle})^{\langle t \rangle}$.
Therefore, we have shown  that $\Gin(U^{\langle j \rangle})=\Gin(U^{\langle j \rangle})^{\langle j \rangle}$ for all $j \geq i$. Now the conclusion follows from Lemma \ref{Afshin} (4).
\end{proof}

\begin{definition}
  Let $M$ be a finitely generated graded $R$-module. We let $\delta(M)=0$ if $M=0$, and we let $\delta(M)$  be the largest graded $R$-submodule of $M$ of dimension at most $\dim(M)-1$ otherwise.   Given $j \geq 0$, we also define the module $\delta^j(M)$ inductively by letting \[\delta^0(M)=M,\,\,\, \delta^1(M) = \delta(M),\,\,\,\text{and}\,\,\, \delta^j(M) = \delta(\delta^{j-1}(M)).\] 
\end{definition}
\noindent Since  $\delta_i(M)$ is the largest submodule of $M$ of dimension at most $i$, it is easy to see that for all $i\in\{0,\ldots,d\}$  there exists $j=j(i) \geq 0$ such that $\delta_i(M) = \delta^j(M)$. In particular $\delta_{d-1}(M)=\delta(M)=\delta^1(M)$.

In the following remark we collect two known facts which are useful in the following.

\begin{remark} \label{remark jsCM}  Recall that in our setting $R\cong \omega_R(n)$. It is a well-known fact that for a $d$-dimensional graded $R$-module $M$ it holds that   $\dim \Ext^{n-i}_R(M,\omega_R)\leq i$ for all $i$.
\begin{enumerate}[(1)]
\item Given a graded submodule  $N\subseteq M$  such that $\dim(N) < \dim(M)=d$, we have that $N=\delta(M)$ if and only if $M/N$ is unmixed of dimension $d$, i.e., $\dim(R/\pp) = \dim(M/N) = d$ for all $\pp \in \Ass(M/N)$. This is a straightforward application of  Proposition \ref{sche02}.
\item A $d$-dimensional finitely generated graded $R$-module $M$ is unmixed if and only if \linebreak $\dim(\Ext^{n-j}_R(M,R)) < j$ for all $j \in \{0,\ldots,d-1\}$. In fact,  if $\pp \in \Ass(M)$ were a prime of height $n-j$ for some $j<d$, we would have that $\Ext^{n-j}_R(M,R)_\pp \cong \Ext^{n-j}_{R_\pp}(M_\pp,R_\pp) \ne 0$, since the latter is, up to shift, the Matlis dual of $H^0_{\pp R_\pp}(M_\pp)$ which is not zero because $\depth(M_\pp)=0$ by \cite[Proposition 1.2.13]{BH93}. It follows that $\pp \in \Supp(\Ext^{n-j}_R(M,R))$ and, thus, $\dim(\Ext^{n-j}_R(M,R))$ $\geq j$, contradiction. The converse is analogous, observing that if $\Ext^{n-j}_R(M,R)$ has dimension at least $j$ and, hence, necessarily equal to $j$, then it must have a prime of  height $n-j$ in its support. 
\end{enumerate}
\end{remark}

The following lemma can be regarded as an enhanced graded version of \cite[Proposition 4.16]{CP17}.

\begin{lemma} \label{lemma finite length} Let $M$ be a finitely generated graded $R$-module. For a sufficiently general homogeneous element $x \in \m$ and for all $j \geq 0$  there is a short exact sequence $0 \to \delta^j(\delta(M)/x\delta(M)) \to \delta^{j+1}(M/xM) \to L_j \to 0$, where $L_j$ is a module of finite length. %In particular, $\dim(\delta^j(\delta(M)/x\delta(M))) = \dim(\delta^{j+1}(M/xM))$.
    
\end{lemma}

\begin{proof}
  %\eee{
    Let $d=\dim(M)$. The case $d \leq 1$ is trivial, therefore we will assume that $d\geq 2$ and proceed by induction on $j \geq 0$.

    \smallskip
    \noindent
    First assume that $j=0$; let $\ov{M} = M/\delta(M)$, and observe that $\ov{M}$ is unmixed of positive depth and dimension $d$. In particular, $\Ext^n_R(\ov{M},R)=0$ and, by Remark \ref{remark jsCM}, we have $\dim(\Ext^{n-\ell}_R(\ov{M},R)) < \ell$ for all $0 < \ell <d $.

    For $x$ sufficiently general, we have that $x$ is $\ov{M}$-regular and, thus, $\delta(M) \cap xM = x(\delta(M):_M x) = x \delta(M)$; moreover,  either $\dim(\delta(M)/x\delta(M)) \leq 0$ or $\dim(\delta(M)/x \delta(M)) = \dim(\delta(M))-1<d-1 = \dim(\ov{M}/x\ov{M})=\dim(M/xM)$. If we let $T=\delta(M)+xM$, then $\delta(M)/x\delta(M) \cong T/xM  \subseteq \delta(M/xM)$, and therefore we have an exact sequence $0 \to \delta(M)/x\delta(M) \stackrel{\varphi}{\to} \delta(M/xM)$. Moreover, we have that 
${\rm coker}(\varphi)\cong \delta((M/xM)/(T/xM)) \cong \delta(M/T)$. Since $M/T\cong \ov{M}/x\ov{M}$, we then have a short exact sequence $$0 \lra \delta(M)/x\delta(M) \stackrel{\varphi}{\lra} \delta(M/xM) \lra \delta(\ov{M}/x\ov{M}) \lra 0.$$
Since $x$ is sufficiently general, by Remark \ref{remark SF -> F} we may assume that $x$ is also strictly filter regular for $\ov{M}$. We then have that either $\dim(\Ext^{n-\ell}_R(\ov{M},R) \otimes_R R/(x)) = \dim(\Ext^{n-\ell}_R(\ov M,R)) \leq 0$ or $\dim(\Ext^{n-\ell}_R(\ov M,R) \otimes_R R/(x)) = \dim(\Ext^{n-\ell}_R(\ov M,R))-1<\ell-1$ for all $0<\ell<d$. Since $x$ is strictly filter regular for $\ov{M}$, from the short exact sequences 
\[
0 \lra \Ext^{n-\ell}_R(\ov M,R) \otimes_R R/(x) \lra \Ext^{n-(\ell-1)}_R(\ov M/x\ov M,R) \lra 0:_{\Ext^{n-(\ell-1)}_R(\ov M,R)} x \lra 0
\]
it also follows that $\dim(\Ext^{n-(\ell-1)}_R(\ov M/x\ov M,R)) = \dim( \Ext^{n-\ell}_R(\ov M,R) \otimes_R R/(x))<\ell-1$ for all $0\leq \ell-1<d-1$. Thus, by Remark \ref{remark jsCM} it follows that $\ov{M}/x\ov{M}$ is unmixed of dimension $d-1$ and  that $L_0=\delta(\ov M/x\ov M)$ is necessarily $\delta_0(\ov M/x\ov M)=H^0_\m(\ov M/x\ov M)$, which  has finite length. %Since $L_0$ has finite length, if $\delta(M/xM)$ has positive dimension then  $\dim(\delta(M)/x\delta(M))=\dim(\delta(M/xM))$.

\medskip
Now suppose that the statement of the lemma is proved for $j-1$, so that we have a short exact sequence $0 \to \delta^{j-1}(\delta(M)/x\delta(M)) \to \delta^j(M/xM) \to L_{j-1} \to 0$,  with $L_{j-1}$ of finite length. %Note that if $\delta^j(M/xM)$ has positive dimension, then $\dim(\delta^j(M/xM)) = \dim(\delta^{j-1}(\delta(M)/x\delta(M)))$.

It is clear from the definition of $\delta$ that there is an exact sequence $0 \lra \delta^j(\delta(M)/x\delta(M)) \stackrel{\varphi_j}{\lra} U \lra L_j \lra 0$, where we let  $U=\delta^{j+1}(M/xM)$ and $L_j={\rm coker}(\varphi_j)$. If $U$ has finite length  we are done, so let us assume that $\dim(U)>0$. In this case we necessarily have that $h=\dim(\delta^j(M/xM))>0$, and since $L_{j-1}$ has finite length we conclude that $\dim(\delta(M)/x\delta(M))= h$. Again because $L_{j-1}$ has finite length, we can find $p \gg 0$ such that $\m^p U \subseteq \delta^{j-1}(\delta(M)/x\delta(M))$. Moreover, $\dim(\m^p U) \leq \dim(U)<h$, and therefore $\m^p U$ is contained in $\delta^{j}(\delta(M)/x\delta(M))$. This shows that $\m^p L_j=0$, and thus $L_j$ has finite length. 
\end{proof}

In \cite{SS17} it is claimed that, if $x \in R$ is $M$-regular, it is possible to prove that $M$ is $i$-sCM if and only if $M/xM$ is $(i-1)$-sCM following the same lines of the  proof \cite[Theorem 4.7]{Sch99}, which is not utterly correct, as we already pointed out in Example \ref{regulare}.  The claim  is indeed false: if $x$ is $M$-regular and $M/xM$ is $(i-1)$-sCM, then $M$ is not necessarily $i$-sCM, as the following example shows.

\begin{example}
Let $R$ be a $2$-dimensional domain which is not Cohen-Macaulay, cf. Example \ref{regulare}. For all $0 \ne x \in R$  we have that $R/(x)$ is $0$-sCM, but $R$ is not even $2$-sCM.
\end{example}
\noindent
In the following proposition we show that the claimed result holds true under some additional assumption.

\begin{proposition} \label{Ale}
 Let $i$ be a positive integer, and $M$ be a finitely generated graded $R$-module with $\depth(M)> 0$. Also assume that $\depth(\Ext^{n-\ell}_R(M,R))>0$ for all $\ell \geq i-1$. For a sufficiently general $x\in R$, if $M/xM$ is $(i-1)$-sCM, then $M$ is $i$-sCM.
\end{proposition}

\begin{proof}
Since a module is $0$-sCM if and only if it is $1$-sCM, cf. Definition \ref{Def i-sCM} and Corollary \ref{coroll H0},  when  $i=1,2$ the statement follows immediately from Proposition \ref{hscor1.9}; therefore we may let $3 \leq i \leq d=\dim(M)$. 
Let $j$ be the smallest integer such that $\delta_{i-1}(M) = \delta^j(M)$, and let $N = M/\delta^j(M)$. Observe that $\depth (N)>0$ by Proposition \ref{assass}. By Remark \ref{remark jsCM-0} it is enough to prove that $N$ is sequentially Cohen-Macaulay, and this is what we do.

  Observe that, since $\dim(\delta^j(M)) \leq i-1$, we have that $\Ext^{n-\ell}_R(M,R) \cong \Ext^{n-\ell}_R(N,R)$ for all $\ell \geq i$ and there is an injection $0 \to \Ext^{n-(i-1)}_R(N,R) \to  \Ext^{n-(i-1)}_R(M,R)$. In particular, from our assumptions we get that
  $$\depth(\Ext^{n-\ell}_R(N,R))>0 \text{\,\, for all \,\,} \ell \geq i-1.$$

  \noindent By a repeated application of Lemma \ref{lemma finite length} there exists a short exact sequence
  $$0 \lra \delta^j(M)/x\delta^j(M)\stackrel{\varphi}{\lra}\delta^j(M/xM) \lra L \lra 0,$$
  where $L$ is a module of finite length.  Now, either  $\dim(\delta^j(M/xM))= \dim(\delta^j(M)/x\delta^j(M)) = \dim(\delta^j(M))-1$, or $\dim(\delta^j(M/xM))\leq 0$, and in both cases we have that $\dim(\delta^j(M/xM))$ $\leq $ $i-2$. By minimality of $j$, we have that $\dim(\delta^{j-1}(M)) > i-1$ and applying iteratively Lemma \ref{lemma finite length} as we did above, we also obtain  $\dim(\delta^{j-1}(M/xM)) = \dim(\delta^{j-1}(M))-1 > i-2$; thus, we may conclude that $$\delta^j(M/xM) = \delta_{i-2}(M/xM).$$
 Since $L={\rm coker}(\varphi)$ and  $x$ is regular for  $N= M/\delta^j(M)$, we have $(M/xM)/(\delta^j(M)/x\delta^j(M))$ $\cong $ $M/(\delta^j(M)+xM) \cong N/xN$, and the above yields a short exact sequence $0 \to L \to N/xN \to (M/xM)/\delta_{i-2}(M/xM) \to 0.$

Since $L$ has finite length, the associated long exact sequence of $\Ext$-modules yields that $$\Ext^{n-\ell}_R(N/xN,R) \cong \Ext^{n-\ell}_R((M/xM)/\delta_{i-2}(M/xM),R) \text{\,\,for all\,\,}  \ell \ne 0.$$ In particular, being $M/xM$ a $(i-1)$-sCM module by assumption, Remark \ref{remark jsCM-0} and Peskine Theorem \ref{peskine} imply that $\Ext^{n-\ell}_R(N/xN,R)$ is either zero or Cohen-Macaulay of dimension $\ell$ for all $\ell\neq 0$. Remark \ref{remark jsCM-0} together with Local Duality also imply that  
\[
\Ext^{n-(\ell-1)}_R(N/xN,R)=0 \text{ for all } 1 \leq \ell-1 \leq i-2.
\]

We may assume that $x$, which  is $N$-regular, is also    $\Ext^{n-\ell}_R(N,R)$-regular for all $\ell \geq i-1$, for we proved above that all these modules  have positive depth.  For $\ell \geq i$ we thus  have short exact sequences $0 \to \Ext^{n-\ell}_R(N,R) \stackrel{\cdot x}{\lra} \Ext^{n-\ell}_R(N,R) \to \Ext^{n-(\ell-1)}_R(N/xN,R) \to 0$, from which  it follows that $\Ext^{n-\ell}_R(N,R)$ is either zero or Cohen-Macaulay of dimension $\ell$ for all $\ell \geq i$.

\noindent
From the above, we also have that for all $2 \leq \ell \leq i-1$ 
the maps $\Ext^{n-\ell}_R(N,R) \stackrel{\cdot x}{\longrightarrow} \Ext^{n-\ell}_R(N,R)$ are isomorphisms  and,  by graded Nakayama's Lemma, that $\Ext^{n-\ell}_R(N,R)=0$ for all $2 \leq \ell \leq i-1$. Finally, we also have an injection $0 \to \Ext^{n-1}_R(N,R) \stackrel{\cdot x}{\longrightarrow} \Ext^{n-1}_R(N,R)$, which implies that $\Ext^{n-1}_R(N,R)$ is Cohen-Macaulay of dimension one by Remark \ref{remark jsCM}.

Applying Peskine's Theorem 2.9, we have thus showed that $N$ is sequentially Cohen-Macaulay, that is, $M$ is $i$-sCM. 

\end{proof}

 The next theorem provides a characterization of partially sequentially Cohen-Macaulay modules. It was proved for the first time in \cite[Theorem 3.5]{SS17} in the ideal case. Here, we generalize the result to finitely generated modules and fix the gap in the original proof thanks to Proposition \ref{Ale}. We let $R_{[n-1]}=k[x_1, \dots, x_{n-1}] \cong R/x_nR$ and denote by $N_{[n-1]}$ the $R_{[n-1]}$-module $N/x_n N \simeq N  \otimes_R R/x_nR$ by restriction of scalars. We let $x\in R$ be a general linear form which, without loss of generality, we may write as $l=a_1x_1 + \dots + a_{n-1}x_{n-1}-x_n$ and consider the map $g_n\: R \to R_{[n-1]}$, defined by $x_i\mapsto x_i$ for $i=1,\ldots,n-1$ and $x_n\mapsto a_1x_1 + \dots + a_{n-1}x_{n-1}$. Then, the surjective homomorphism $F/U \rightarrow F_{[n-1]}/g_n(U)$
has kernel $(U+xF)/U$ and induces the isomorphism  
\begin{equation}\label{generic}
\frac{F}{U+xF} \cong \frac{F_{[n-1]}}{g_n(U)}.
\end{equation}
Moreover, the image of $\Gin(U)$ in $F_{[n-1]}$ via the mapping $x_n \mapsto 0$ is $\Gin(U)_{[n-1]}$. With this notation, the module version of \cite[Corollary 2.15]{Gre98} states that 
\begin{equation}\label{215}
\Gin(g_n(U))=\Gin(U)_{[n-1]}.
\end{equation}

\begin{theorem}\label{partiallySCM}
Let $M$ be a finitely generated graded $R$-module of dimension $d$, and let $M \cong F/U$ be a free graded presentation of $M$. The following conditions are equivalent:
\begin{enumerate}[(1)]
\item $F/U$ is $i$-sCM;
\item $h^j(F/U)=h^j(F/\Gin(U))$ for all $i\leq j\leq d$.
\end{enumerate}
\end{theorem}

\begin{proof}
(1) $\Rightarrow$ (2) is a direct consequence of Lemma \ref{lemma SbSt} (3) and (4), since also $F/\Gin(U)$ is $i$-sCM.
    
\smallskip
\noindent
We prove the converse by induction on $d$. If $d=0$, $F/U$ is Cohen-Macaulay and  sequentially Cohen-Macaulay. Therefore, without loss of generality we may assume that $F/U$ and $F/\Gin(U)$ have positive dimension and, by Lemma 
\ref{lemma SbSt} (1), also positive depth.
Since $F/\Gin(U)$ is sequentially Cohen-Macaulay, Peskine's Theorem \ref{peskine} implies that there exists a linear form $l \in R$ which is $F/\Gin(U)$-regular and also regular for all non-zero $\Ext^{n-j}_R(F/\Gin(U),\omega_R)$ with $j>0$. 
Starting with the exact sequence $0 \rightarrow F/\Gin(U) (-1) \stackrel{\cdot l}\rightarrow F/\Gin(U) \rightarrow F/(\Gin(U)+lF) \rightarrow 0$, by the above and Local Duality we obtain the short exact sequences
$$
0 \lra H^{j-1}_{\mm}(F/(\Gin(U)+lF)) \lra H^j_{\mm}(F/\Gin(U))(-1) \stackrel{\cdot l}\lra H^j_{\mm}(F/\Gin(U)) \lra 0,
$$ 
 from which it follows that $h^{j-1}(F/(\Gin(U)+lF))=(z-1)h^j(F/\Gin(U))$ for all $j$.

Consider now a sufficiently general linear form $x \in R$. For all $j$, there are exact sequences 
$$
0 \lra B^{(j)} \lra H^{j-1}_{\mm}(F/(U+xF)) \lra H^j_{\mm}(F/U)(-1) \lra H^j_{\mm}(F/U) \lra C^{(j)} \lra 0
$$
for some $R$-modules $B^{(j)}$ and $C^{(j)}$, and these imply that $h^{j-1}(F/(U+xF)) = (z-1)h^j(F/U)+\Hilb(B^{(j)})+\Hilb(C^{(j)})$ for all $j$.
By \eqref{generic} and \eqref{215}, we obtain
\begin{equation*}\begin{split}
(z-1)h^j(F/U) &\leq (z-1)h^j(F/U)+\Hilb(B^{(j)})+\Hilb(C^{(j)}) \\
&=h^{j-1}(F/(U+xF))=h^{j-1}(F_{[n-1]}/g_n(U))  \\
& \leq  h^{j-1}(F_{[n-1]}/\Gin(g_n(U))) 
=h^{j-1}(F_{[n-1]}/\Gin(U)_{[n-1]}) \\
&= h^{j-1}(F/(\Gin(U)+lF))=(z-1)h^j(F/\Gin(U)).
\end{split}
\end{equation*}
Thus, from our hypothesis it follows that the above inequalities are equalities for all $j \geq i$, that  means that $\Hilb(B^{(j)})=\Hilb(C^{(j)})=0$, i.e. $B^{(j)}=C^{(j)}=0$ for  $j\geq i$. Moreover, since $C^{(i-1)}=B^{(i)}=0$, it follows that $x$ is regular for all non-zero $\Ext_R^{n-j}(F/U,\omega_R)$ with  $j\geq i-1$, which thus have positive depth.

\noindent
From the above equalities we also get that $h^{j}(F_{[n-1]}/g_n(U)) = h^{j}(F_{[n-1]}/\Gin(g_n(U)))$ for all $j \geq i-1$; by induction, this implies that $F/(U+xF)\cong F_{[n-1]}/g_n(U) $ is $(i-1)$-sCM.  

\noindent The conclusion follows now by a straightforward application of Proposition \ref{Ale}.
\end{proof}

As a corollary, we immediately obtain Theorem \ref{HS Intro}.
%, which is a characterization of sequentially Cohen-Macaulay modules due to Herzog and the third author \cite[Theorem 3.1]{HS01}. 

\begin{theorem}\label{HeSb}
Let $M$ be a finitely generated graded $R$-module, and let $M \cong F/U$ a free graded presentation of $M$. Then, $F/U$ is sequentially Cohen-Macaulay if and only if $h^i(F/U)=h^i(F/\Gin(U))$ for all $j \geq 0$. 
\end{theorem}

\noindent
One can  wonder whether the equality $h^i(F/U)=h^i(F/\Gin(U))$ is enough to imply that $F/U$ is $i$-sCM; however, this is not the case.

\begin{example}\label{giuliospezza} %1)\eee{
\begin{enumerate}[(1)]
\item Consider a graded Cohen-Macaulay $k[x_1,\dots,x_n]$-module $M_1$ of dimension $i$  and a graded non-sequentially Cohen-Macaulay $k[x_{i+2},\dots,x_n]$-module $N$. Let $M_2=N\otimes_k k[x_1,\ldots,x_{i+1}]$ and take $M=M_1 \oplus M_2$. Then $x_1,\ldots,x_{i+1}$ is a strictly filter-regular sequence for $M_2$, and it follows from Proposition \ref{hscor1.9} (1) that $M_2$ is not sequentially Cohen-Macaulay. By Corollary \ref{sumviceversa} also $M$ is not sequentially Cohen-Macaulay. On the other hand, since $\depth(M_2) > i$ we have that $H^i_\m(M) \cong H^i_\m(M_1)$. If we write $M_1 \cong F_1/U_1$ and $M_2 \cong F_2/U_2$, where $F_1$ and $F_2$ are graded free $R$-modules and $U_1,U_2$ are graded submodules, then $M_1 \oplus M_2 \cong F/U$ where $F=F_1 \oplus F_2$ and $U = U_1 \oplus U_2$, and it follows that $\Gin(U) = \Gin(U_1) \oplus \Gin(U_2)$. Since $M_1$ is Cohen-Macaulay, hence sequentially Cohen-Macaulay, and $\depth(F_2/\Gin(U_2)) = \depth(M_2)>i$ we therefore conclude by Theorem \ref{HeSb} that $h^i(F/U) = h^i(F_1/U_1) = h^i(F_1/\Gin(U_1)) = h^i(F/\Gin(U))$. %} \ooo{ it is enough to consider the direct sum $M \oplus (N\otimes_k k[x_1,\dots,x_{i+1}])$ where $M$ is a Cohen-Macaulay $k[x_1,\dots,x_n]$-module of dimension $i$ and $N$ is a not sequentially Cohen-Macaulay $k[x_{i+2},\dots,x_n]$-module.}
\item  The following is another explicit example of such instance in the ideal case.
  
\noindent Consider the polynomial ring  $R=k[x_1,x_2,x_3,x_4,x_5,x_6,x_7]$ and the monomial ideal $I=(x_1^3, x_1^2 x_2 x_4, x_1x_5, x_1x_6, x_2x_5, x_2x_6, x_2^2 x_7^2, x_3x_5, x_3x_6, x_3x_7, x_4x_5, x_4 x_6, x_4 x_7 , x_7^3)$. Then, one can check that $\depth(R/I)=0$ and $\dim(R/I)=3$; moreover,  $h^j(R/I)=h^j(R/\Gin(I))$ for $j=0,3$, and   $h^j(R/I)\neq h^j(R/\Gin(I))$ for $j=1,2$. By Theorem \ref{partiallySCM}, this means that $R/I$ is $3$-sCM but not $2$-sCM and, {\em a fortiori}, not $1$-sCM.
\end{enumerate}
\end{example}

\noindent
On the other hand, if $I\lex$ is the lexicographic ideal associated with $I$, then the equality $h^i(R/I) = h^i(R/I\lex)$ ensures the $i$-partial sequential Cohen-Macaulayness of $R/I$. Notice that this is a stronger condition, though, since  $h^i(R/I) \leq h^i(R/\Gin(I)) \leq h^i(R/I\lex)$, coefficientwise, see \cite[Theorems 2.4 and 5.4]{Sb01}. Actually, in \cite[Theorem 4.4]{SS17} the following result is proved.

\begin{theorem}\label{lbutnol}
Let $i$ be a positive integer and $I$ a homogeneous ideal of $R$; then, the following conditions are equivalent:
\begin{enumerate}[(1)]
\item $h^i(R/I)=h^i(R/I\lex)$; 
\item $h^i(R/\Gin(I))=h^i(R/I\lex)$; 
\item $h^j(R/I)=h^j(R/I\lex)$ for all $j\geq i$.
\end{enumerate}
If any of the above holds, then $I$ is $i$-sCM.
\end{theorem}

We conclude this section by observing that the conditions in the previous theorem are still equivalent if we replace $\Gin(I)$ with $\Gin_0(I)$, the zero-generic initial ideal of $I$ introduced in \cite{CS15}.
We also remark that the equivalence between conditions (2) and (3)  is not true if we replace the Hilbert series of local cohomology modules with graded Betti numbers, see \cite[Theorem 3.1]{MH06}.

\section{E-depth} \label{Section Edepth}
As in the previous section,  $k$ will denote an infinite field, $R=k[x_1,\ldots,x_n]$ a standard graded polynomial ring and $\m= (x_1,\ldots,x_n)$ its homogeneous maximal ideal. As before, when $M=0$ we let $\dim(M)=-1$ and $\depth(M)=\infty$.  We start with the main definition of this section.

\begin{definition} \label{Def E-depth} 
Given a non-zero finitely generated graded $\ZZ$-module $M$ and a non-negative integer $r$, we say that $M$ satisfies condition $(E_r)$ if  $\depth(\Ext^i_R(M,R)) \geq \min\{r,n-i\}$ for all $i \in \ZZ$. We let
\[
\Edepth(M) = \min\left\{n, \sup\{r \in \NN \mid M \text{ satisfies } (E_r)\}\right\}.
\]
\end{definition}

\begin{remark} \label{Remark E-depth}
Observe that a non-zero module $M$ is sequentially Cohen-Macaulay if and only if it satisfies condition $(E_r)$ for all $r \geq 0$, see Theorem \ref{peskine}. It is therefore clear that a sequentially Cohen-Macaulay $R$-module has $\Edepth$ equal to $n$. The converse is also true, since if $\Edepth(M)=n$, then $\depth(\Ext^i_R(M,R)) \geq n-i$; as $\dim(\Ext^i_R(M,R)) \leq n-i$ always holds, cf. Remark \ref{remark jsCM}, this implies the claim.
\end{remark}

\begin{lemma} \label{lemma ses Edepth} Let $M \ne 0$ be a finitely generated $\ZZ$-graded $R$-module with positive depth and E-depth, and let $\ell$ be a linear form which is a strictly filter regular for $M$; then, the graded short exact sequence $0 \lra M(-1) \stackrel{\cdot \ell}{\lra} M \lra M/\ell M \lra 0$ induces graded short exact sequences 
\[
0 \lra  \Ext^i_R(M,R) \stackrel{\cdot \ell}{\lra}   \Ext^i_R(M,R)(1) \lra \Ext^{i+1}_R(M/\ell M,R) \lra 0,\,\,\, \text{ for all } i<n,
\]
and 
\[
0 \lra H^{i-1}_\m(M/\ell M) \lra  H^i_\m(M)(-1) \stackrel{\cdot \ell}{\lra}  H^i_\m(M) \lra 0,\,\,\, \text{ for all } i>0.
\]
\end{lemma}
\begin{proof}
Consider the induced long exact sequence of $\Ext$ modules
\[
\xymatrix{
\ldots \ar[r] & \Ext^i_R(M,R) \ar[r]^-{\cdot \ell} & \Ext^i_R(M(-1),R) \ar[r] & \Ext^{i+1}_R(M/\ell M,R) \ar[r] & \ldots \,\,,
}
\]
and observe that $\Ext^i_R(M(-1),R) \cong \Ext^i_R(M,R)(1)$. Our assumption that $\Edepth(M)>0$ guarantees that $\depth(\Ext^i_R(M,R))>0$ for all $i<n$. Since $\ell$ is  strictly filter regular  for $M$, the multiplication by $\ell$ is injective on $\Ext^i_R(M,R)$ for all $i<n$, and the long exact sequence breaks into short exact sequences, as claimed, and the graded short exact sequences of local cohomology modules are obtained by graded Local Duality. 
\end{proof}

\begin{proposition} \label{prop Edepth} Let $M$ $\neq 0$ be a finitely generated $\ZZ$-graded $R$-module.
\begin{enumerate}[(1)]
\item $\Edepth(M) = \Edepth(M/H^0_\m(M))$.
\item Assume that $\Edepth(M)>0$  and $\ell$ is  homogeneous strictly filter regular for $M$; then, either $\Edepth(M/(H^0_\m(M)+\ell M)) = \Edepth(M)=n$, or $\Edepth(M/(H^0_\m(M)+\ell M)) = \Edepth(M)-1$.
\end{enumerate}
\end{proposition}
\begin{proof}
We first prove (1); clearly, we may assume that $H^0_\m(M) \ne 0$. Applying the functor $\Hom_R(-,R)$ to the short exact sequence $0 \lra H^0_\m(M) \lra M \lra M/H^0_\m(M) \lra 0$ we get a long exact sequence of $\Ext$ modules, from which we obtain that $\Ext^{n-i}_R(M,R) \cong \Ext^{n-i}_R(M/H^0_\m(M),R)$ for all $i \ne 0$ because $H^0_\m(M)$ has finite length; the first statement is now clear, since  $\depth(\Ext^n_R(H^0_\m(M),R))=0$ .

\medskip
For the proof of (2), let $N=M/H^0_\m(M)$; by (1), $M$ and $N$ have same E-depth and $\Ext^{i}_R(M,R) \cong \Ext^{i}_R(N,R)$ for all $i\neq n$. First assume that $M$ is sequentially Cohen-Macaulay, i.e., $\Edepth(M)=n$. Then, $N$ is sequentially Cohen-Macaulay by Remark \ref{Remark E-depth}  or by Corollary \ref{coroll H0}, and since $\ell$ is filter regular for $M$ by Remark \ref{remark SF -> F}, $\ell$ is $N$-regular. It follows from  Proposition \ref{hscor1.9} that $M/(H^0_\m(M) + \ell M)$ is sequentially Cohen-Macaulay and, therefore, has E-depth equal to $n$.

\noindent Now assume that $0<\Edepth(N)=\Edepth(M) = r<n$; from an application of  Lemma \ref{lemma ses Edepth}  we obtain short exact sequences $0 \to \Ext^i_R(N,R) \stackrel{\cdot \ell}{\to} \Ext^i_R(N,R)(1) \to \Ext^{i+1}_R(N/\ell N,R) \to 0$  and,   therefore, $\Ext^{i+1}_R(N/\ell N,R) \cong \Ext^i_R(M,R)(1)/\ell \Ext^i_R(M,R)$ for all $i<n$. It follows that $\depth(\Ext^{i+1}_R(N/\ell N,R)) = \depth(\Ext^i_R(M,R))-1$ for all $i<n$, which clearly implies  $\Edepth(N/\ell N) = \Edepth(M)-1$, as desired.
\end{proof}

We now introduce a special grading on a polynomial ring  which refines the standard grading, and that can be further refined to the monomial $\ZZ^n$-grading.

\begin{definition}
 Let $r$ be a positive integer, $A$ be a $\ZZ$-graded ring and  $S=A[y_1,\ldots,y_r]$ a polynomial ring over $A$. For $i \in \{0,\ldots,r\}$  let $\eta_i \in \ZZ^{r+1}$ be the vector whose $(i+1)$-st entry is $1$, and all other entries equal $0$. We consider $S$ as a $\ZZ \times \ZZ^r$-graded ring by letting
    \[
    \deg_S(a) = \deg_A(a) \cdot \eta_0,\,\,\text{for all }\,  a \in A\,\,\, \text{ and }\,  \deg_S(y_i) = \eta_i, \,\,\text{ for } i \in \{1,\ldots,r\}.
  \]

\end{definition}
By means of the previous definition,  we may consider $R=k[x_1,\ldots,x_n]$ as a $\ZZ \times \ZZ^r$-graded ring for any $0 \leq r \leq n-1$ by letting $\deg_R(x_i) = \eta_0$ for all $1 \leq i \leq n-r$ and $\deg_R(x_i) = \eta_i$ for all $n-r+1 \leq i \leq n$. Observe that an element $f \in R$ is graded with respect to such grading if and only if $f$ can be written as $f=\overline{f} \cdot u$, where $\overline{f} \in k[x_1,\ldots,x_{n-r}]$ is homogeneous with respect to the standard grading, and $u$ is a monomial in $k[x_{n-r+1},\ldots,x_n]$. In particular, when $r=0$ this is just the standard grading on $R$, while for $r=n-1$ it coincides with the monomial $\ZZ^n$-grading. 

\begin{remark} We can extend in a natural way such a grading  to free $R$-modules $F$ with a basis by assigning degrees to the elements of the given basis.  Accordingly, any $R$-module $M$ is $\ZZ \times \ZZ^r$-graded if and only if $M \cong F/U$, where $F$ is a free $\ZZ \times \ZZ^r$-graded $R$-module and $U$ is a $\ZZ \times \ZZ^r$-graded submodule of $F$.
\end{remark}

We see next that the grading just introduced  is very relevant to the purpose of estimating the E-depth of a module; the following can be seen as  a refinement of Proposition \ref{wsscm}.
\begin{proposition} \label{Edepth graded}
Let $R=k[x_1,\ldots,x_n]$, and $M$ be a finitely generated $\ZZ \times \ZZ^r$-graded $R$-module such that $x_n,\ldots,x_{n-r+1}$ is a filter regular sequence for $M$. Then, $x_n,\ldots,x_{n-r+1}$ is a strictly filter regular sequence for $M$, and $\Edepth(M) \geq r$.
\end{proposition}
\begin{proof}
  Write $M$ as $M=F/U$, where $F$ is a finitely generated $\ZZ \times \ZZ^r$-graded free $R$-module, and $U$ is a $\ZZ \times \ZZ^r$-graded submodule of $F$. Since $U$ is $\ZZ \times \ZZ^r$-graded, $x_n$ is $\ZZ\times \ZZ^r$-homogeneous and filter regular, $U\sat= U:_F x_n^{\infty}$ is also $\ZZ\times \ZZ^r$ graded and, accordingly $M/H^0_\m(M) \cong F/U\sat$ is too. If $F/U\sat=0$, then $M$ has dimension zero and, thus, is  sequentially Cohen-Macaulay; then  $\Edepth(M) = n \geq t$, and every sequence of non-zero elements of $(x_1,\ldots,x_n)$ is a strictly filter regular sequence for $M$. In particular, $x_n,\ldots,x_{n-r+1}$ is a strictly filter regular sequence for $M$.

Now suppose that $F/U\sat \ne 0$. By assumption, $x_n$ is filter regular for $M$, and thus regular for $F/U\sat$. Since the latter is $\ZZ \times \ZZ^r$-graded, we have that $F/U\sat \cong \ov{F}/\ov{U} \otimes_k k[x_n]$ for some $\ZZ \times \ZZ^{r-1}$-graded $\ov{R}=k[x_1,\ldots,x_{n-1}]$-module $\ov{F}$, and some $\ZZ\times \ZZ^{r-1}$-graded submodule $\ov{U}$ of $\ov{F}$ such that $\ov{F}/\ov{U}$ can be identified with the hyperplane section $F/U\sat \otimes_R R/(x_n)$. Thus, for $i<n$ we have
\[
\Ext^i_R(M,R) \cong \Ext^i_R(M/H^0_\m(M),R) \cong \Ext^i_R(F/U\sat,R) \cong \Ext_{\ov{R}}^i(\ov{F}/\ov{U},\ov{R}) \otimes_k k[x_n].
\]
It follows that $x_n$ is a non-zero divisor on $\Ext^i_R(M,R)$ for all $i<n$, and thus $\Edepth(M)>0$. Since $\Ext^n_R(F/U,R)$ has finite length it also follows that $x_n$ is a strictly filter regular element for $M$. Now we can consider $\ov{F}/\ov{U}$, and an iteration of this argument will imply the desired conclusion.
\end{proof}

We now introduce a weight order on $R$. Given integers $0 \leq r \leq n$, consider the following $r \times n$ matrix 
\[
\Omega_{r,n} = \begin{bmatrix} 0 & 0 & \ldots & 0 & 0 & 0 & 0 & \ldots & 0 & -1 \\
0 & 0 & \ldots & 0 & 0 & 0 & 0 & \ldots & -1 & 0 \\
\vdots & \vdots &  & \vdots & \vdots & \vdots &  \vdots & \iddots & \vdots &\vdots\\
0 & 0 & \ldots & 0 & 0 & 0 & -1 & \ldots & 0 & 0 \\
0 & 0 & \ldots & 0 & 0 & -1 & 0 & \ldots & 0 & 0 \\
0 & 0 & \ldots & 0 &-1 & 0 & 0 & \ldots & 0 & 0
\end{bmatrix},
\]
%\vskip -.7cm
%\hspace{4.7cm}\underbrace{
%\begin{matrix}
%\phantom{aaaaaaaaaa& 0} \end{matrix}}_{n-r}
%}
%\underbrace{
%\begin{matrix}
%\phantom{aaaaaaaaaaaaaaaaaaaaaa  & 0 
%}\end{matrix}}_{r}
%}

\smallskip    
\noindent
and let $\omega_i$ be its $i$-th row; then,  this induces a ``partial'' revlex order $\rev_r$ on $R$ by declaring that a monomial $x^{\ul{a}} = x_1^{a_1} \cdots x_n^{a_n}$ is greater than a monomial $x^{\ul b} = x_1^{b_1} \cdots x_n^{b_n}$ if and only if there exists $1 \leq j \leq r$ such that $\omega_i \cdot \ul{a} = \omega_i \cdot \ul{b}$ for all $i \leq j$ and $\omega_{j+1} \cdot \ul{a} > \omega_{j+1} \cdot \ul{b}$.

\begin{remark} \label{remark revlex} Observe that $\rev_n$ coincides with the usual revlex order on $R=k[x_1,\ldots,x_n]$. %\ooo{Observe that an initial ideal with respect to the order $\rev_n$ coincides with an initial ideal with respect to the usual revlex order on $R=k[x_1,\ldots,x_n]$.}
\end{remark}

Give  a graded free $R$-module $F$ with basis $\{e_1,\ldots,e_s\}$,  we can write any $f \in F$ uniquely as a finite sum $f =\sum u_j e_{i_j}$, where $u_j$ are monomials and we assume the sum has minimal support. We let the initial form ${\rm in}_{\rev_r}(f)$ of $f$ written as above  be the sum of those $u_j e_{i_j}$ for which $u_j$ is maximal with respect to the order $\rev_r$ introduced above.

\begin{definition} \label{def partial gin}
Let $F$ be a finitely generated graded free $R$-module, and $U$ be a graded submodule; let also $M=F/U$. We say that {\em the $r$-partial general initial submodule of $U$ satisfies a given property} {\bf {$\mathcal P$}} if there exists a non-empty Zariski open set $\mathcal L$ of $r$-uples of linear forms such that $F/{\rm in}_{\rev_r}(g_\ell(U))$ satisfies  {\bf {$\mathcal P$}} for any $\ell=(\ell_{n-r+1},\ldots,\ell_n) \in \mathcal{L}$, where $g_{\ell}$ is the automorphism on $F$ induced by the change of coordinates of $R$ which sends $\ell_i  \mapsto x_i$ and fixes the other variables.
\end{definition} 

With some abuse of notation,  we denote any partial initial submodule ${\rm in}_{\rev_r}(g_\ell(U))$ which satisfies  {\bf {$\mathcal P$}} {\it the} $r$-partial general submodule of $U$, and denote it by $\gin_r(U)$.

\begin{proposition} \label{prop gin and grading}
Let $F$ be a finitely generated graded free $R$-module, and $U \subseteq F$ be a graded submodule. For every $0 \leq r \leq n$ we have that $\Edepth(F/\gin_r(U)) \geq r$.
\end{proposition}
\begin{proof}
For a sufficiently general change of coordinates we have that $x_n,\ldots,x_{n-r+1}$ form a filter regular sequence for $F/(g_\ell(U))$. Since ${\rm in}_{\rev_r}(U):_Fx_n = {\rm in}_{\rev_r}(U:_Fx_n)$, see \cite[15.7]{Eis95}, we have that $x_n,\ldots,x_{n-r+1}$ also form a filter regular sequence for $F/{\rm in}_{\rev_r}(g_\ell(U))$, and hence for $F/\gin_r(U)$. By construction, the module $F/\gin_r(U)$ is $\ZZ \times \ZZ^r$-graded, and the claim now follows from Proposition \ref{Edepth graded}.
\end{proof}

\begin{theorem} \label{main thm Edepth}
Let $F$ be a finitely generated graded free $R$-module, and $U \subseteq F$ be a graded submodule. For every $0 \leq r \leq n$ we have that $\Edepth(F/U) \geq r$ if and only if $h^i(F/U) = h^i(F/\gin_r(U))$ for all $i \in \NN$.
\end{theorem}

\begin{proof}
After performing a sufficiently general change of coordinates, we may assume that $V={\rm in}_{\rev_r}(U)$ has the same properties as $\gin_{\rev_r}(U)$, and that $x_n,\ldots,x_{n-r+1}$ form a strictly filter regular sequence for $F/U$. By using again that ${\rm in}_{\rev_r}(U:_Fx_n) = {\rm in}_{\rev_r}(U):_F x_n$, and because strictly filter regular sequences are filter regular by Remark \ref{remark SF -> F}, we have that  $x_n,\ldots,x_{n-r+1}$ form a filter regular sequence for $F/V$. Since $V$ is $\ZZ \times \ZZ^r$-graded, it follows from Proposition \ref{Edepth graded} that $x_n,\ldots,x_{n-r+1}$ is a strictly filter regular sequence for $F/V$. By \cite[15.7]{Eis95} we also have that
\[
V\sat = {\rm in}_{\rev_r}(U):_F x_n^\infty = {\rm in}_{\rev_r}(U:_F x_n^\infty) = {\rm in}_{\rev_r}(U\sat),
\]
and 
\[
V\sat+x_nF = {\rm in}_{\rev_r}(U\sat) + x_nF = {\rm in}_{\rev_r}(U\sat+x_nF).
\]
Viewing $F/(U\sat+x_nF)$ as a quotient of a free $\ov{S}=k[x_1,\ldots,x_{n-1}]$-module $\ov{F}$ by a graded submodule $\ov{U}$, we see that $V\sat+x_nF$ can be identified with a submodule $\ov{V} \subseteq \ov{F}$, with $\ov{V} = {\rm in}_{\rev_{r-1}}(\ov{U})$ and $\ov{V}$ has the same properties of $\gin_{\rev_{t-1}}(\ov{U})$, see \cite[Lemma 3.4]{CD21} and the proof of \cite[Theorem 3.6]{CD21} for more details.

First assume that $\Edepth(F/U) \geq r$, and we prove equality for the Hilbert series of $H^i_\m(F/U)$ and $H^i_\m(F/V)$ for all $i \in \NN$ by induction on $r \geq 0$. The base case is trivial since ${\rm in}_{\rev_0}(U) = U$. By Proposition \ref{prop Edepth} we have that $\Edepth(F/U\sat) \geq r > 0$ and, accordingly,   $\Edepth(F/(U\sat + x_nF)) \geq r-1$. By induction, we have that $H^i_\m(F/(U\sat+x_nF)) = H^i_\m(F/(V\sat+x_nF))$. Since $x_n$ is strictly filter regular for both $F/U\sat$ and $F/V\sat$, and these have positive $\Edepth$, by Lemma \ref{lemma ses Edepth} we have graded short exact sequences for all $i>0$:
\begin{equation}
\label{eq1}
\xymatrix{
0 \ar[r] & H^{i-1}_\m(F/(U\sat+x_nF)) \ar[r] & H^i_\m(F/U\sat)(-1) \ar[r]^-{\cdot x_n} & H^i_\m(F/U\sat) \ar[r] & 0,
}
\end{equation}
and
\begin{equation}
\label{eq2}
\xymatrix{
0 \ar[r] & H^{i-1}_\m(F/(V\sat+x_nF)) \ar[r] & H^i_\m(F/V\sat)(-1) \ar[r]^-{\cdot x_n} & H^i_\m(F/V\sat) \ar[r] & 0.
}
\end{equation}
Since the first modules in both sequences have the same Hilbert series, a straightforward computation shows that $h^i(F/U\sat) = h^i(F/V\sat)$. Since $H^i_\m(F/U\sat) \cong H^i_\m(F/U)$ and $H^i_\m(F/V\sat) \cong H^i_\m(F/V)$ for all $i>0$, the equality between Hilbert series is proved for $i>0$. Finally, since $\Hilb(F/U) = \Hilb(F/V)$ and $\Hilb(F/U\sat) = \Hilb(F/V\sat)$, we have that
\[
h^0(F/U) = \Hilb(U\sat/U) = \Hilb(V\sat/V) = h^0(F/V).
\]
Conversely, assume that the local cohomology modules of $F/U$ and $F/V$ have the same Hilbert series. If $r=0$ there is nothing to show, otherwise it is enough to prove that $\Edepth(F/U\sat) \geq r$.

Since $x_n$ is a filter regular element for $F/V$, it is strictly filter regular for $F/V$, being the latter $\ZZ \times \ZZ^r$-graded, and hence it is filter regular for $F/V\sat$. Thus, the sequence (\ref{eq2}) is exact for $i>0$. On the other hand, suppose by way of contradiction that the sequence (\ref{eq1}) is not exact for some $i \in \ZZ$, and let $i$ be the smallest such integer, so that we still have an exact sequence
\[
\xymatrix{
0 \ar[r] & H^{i-1}_\m(F/(U\sat+x_nF)) \ar[r] & H^i_\m(F/U\sat)(-1) \ar[r]^-{\cdot x_n} & H^i_\m(F/U\sat).
}
\]
Counting dimensions in such a sequence, and comparing them to those obtained from (\ref{eq2}) we obtain that $h^{i-1}(F/(U\sat+x_nF))_j>h^{i-1}(F/(V\sat+x_nF))_j$. However, by upper semi-continuity we know that the reverse inequality always holds,  which gives a contradiction. Thus, the sequence (\ref{eq1}) is exact for all $i \in \NN$. %Viewing $F/(U^{\sat}+x_nF)$ as a quotient of a free $\ov{S}=k[x_1,\ldots,x_{n-1}]$-module $\ov{F}$ by a graded submodule $\ov{U}$, we see that $V^{\sat}+x_nF$ can be identified with a submodule $\ov{V} \subseteq \ov{F}$, with $\ov{V} = {\rm in}_{\rev_{r-1}}(\ov{U}) = \ov{V}$ and $\ov{V}$ has the same properties of $\gin_{\rev_{t-1}}(\ov{U})$ (see \cite[Lemma 3.4]{CD21} and the proof of \cite[Theorem 3.6]{CD21} for more details). 
By induction we have that $\Edepth(F/(U\sat+x_nF)) \geq r-1$. If $\Edepth(F/(U\sat+x_nF)) = n$, by Proposition \ref{prop Edepth} we see that $\Edepth(F/U) = n \geq r$, as desired. Otherwise, again by Proposition \ref{prop Edepth} we have that $\Edepth (F/U)  = \Edepth(F/(U\sat +x_nF)) +1 \geq r$.
\end{proof}

\noindent Recalling Remark \ref{remark revlex}, Theorem \ref{main thm Edepth} can be viewed as another extension of Theorem \ref{HS Intro}. 
\bibliographystyle{plain}
\bibliography{References}

\end{document}